\documentclass[11pt,reqno]{amsart}
\usepackage{amssymb,mathrsfs,color,mathtools,empheq, verbatim, epstopdf}
\usepackage{pinlabel}
\mathtoolsset{showonlyrefs}
\usepackage{hyperref} 

\usepackage{enumerate, bbm}

\usepackage{graphicx}
\usepackage{tensor}
\usepackage{slashed}
\usepackage{cite}
\usepackage[dvipsnames]{xcolor}

\usepackage[shortlabels]{enumitem}

\usepackage{geometry}

\numberwithin{equation}{section}
\newtheorem{mainthm}{Theorem}
\newtheorem{thm}{Theorem}[section]
\newtheorem{cor}[thm]{Corollary}
\newtheorem{lem}[thm]{Lemma}
\newtheorem{prop}[thm]{Proposition}

\newtheorem{con}{Conjecture}

\theoremstyle{definition}

\theoremstyle{remark}

\def\bC {\mathbb{C}}
\def\bN {\mathbb{N}}

\def\bR {\mathbb{R}}
\def\bS {\mathbb{S}}

\def\cA {\mathcal{A}}
\def\cB {\mathcal{B}}

\def\cP {\mathcal{P}}

\def\cR {\mathcal{R}}

\newcommand{\la}{\langle}
\newcommand{\ra}{\rangle}

\renewcommand{\div}{\operatorname{div}}

\newcommand{\spn}{\operatorname{span}}
\newcommand{\supp}{\operatorname{supp}}
\newcommand{\Tr}{\operatorname{Tr}}

\newcommand{\Id}{\operatorname{Id}}

\newcommand{\eee}{\mathrm e}
\newcommand{\iii}{\mathrm i}

\definecolor{deepgreen}{cmyk}{1,0,1,0.5}

\newcommand{\C}{\mathbb{C}}

\newcommand{\R}{\mathbb{R}}

\newcommand{\om}{\omega}
\newcommand{\lam}{\lambda}
\newcommand{\te}{\theta}

\newcommand{\ta}{\tau}

\makeatletter

\newcommand{\Rmnum}[1]{\expandafter\@slowromancap\romannumeral #1@}
\makeatother

\newcommand{\EQ}[1]{\begin{equation}\begin{split} #1 \end{split}\end{equation}}

\newcommand{\Del}[1]{}

\renewcommand\Re{\mathrm{Re}\,}

\definecolor{green}{rgb}{0.3,0.8,0} % Redefines the color green.
\newcommand{\tr}{\textrm{tr}}

\renewcommand{\Re}{\mathrm{Re}}

\newcommand{\ud}{\mathrm{d}}

\newcommand{\0}{\emptyset}

\newcommand{\alp}{\alpha}

\newcommand{\veps}{\varepsilon}

\newcommand{\bbC}{\mathbb C}

\newcommand{\bbN}{\mathbb N}

\newcommand{\bbR}{\mathbb R}
\newcommand{\bbS}{\mathbb S}

\newcommand{\calD}{\mathcal D}

\newcommand{\calN}{\mathcal N}

\newcommand{\calP}{\mathcal P}

\newcommand{\calR}{\mathcal R}

\newcommand{\Ric}{\operatorname{Ric}}
\newcommand{\can}{\mathrm{can}}
\newcommand{\vol}{\operatorname{vol}}
\newcommand{\dvol}{\,\mathrm{dvol}}
\newcommand{\CD}{\operatorname{CD}}

\begin{document}
\title[Spectral obstructions to contracting transport maps on curved spaces]{Spectral obstructions to contracting transport maps on curved spaces}
\author{Shrey Aryan}

\begin{abstract}
Caffarelli's contraction theorem states that the Brenier optimal transport
map from the standard Gaussian measure to a more log-concave probability
measure is $1$-Lipschitz. Motivated by this theorem,
Milman~\cite{milman2018spectral} formulated several conjectures for the
round sphere and for weighted manifolds satisfying the curvature-dimension
condition $\CD(\rho,\infty)$.  A contracting transport map as in these conjectures implies a corresponding spectral comparison. In the spherical setting, this
comparison was also conjectured by Colding and Minicozzi
\cite{colding1998weyl} for compact manifolds with Ricci curvature lower
bounds. We construct counterexamples to these conjectured spectral
comparisons on spheres in dimensions $d\geq4$ and on smooth complete weighted manifolds
diffeomorphic to $\R^d$ satisfying the $\CD(1,\infty)$ condition in dimensions $d\geq4$, thereby obtaining obstructions to contracting transport maps. The spherical counterexamples can be chosen arbitrarily close to the unit round metric in $C^\infty$, while retaining $\Ric_g\geq(d-1)g$.  The weighted counterexamples can be chosen with non-negative sectional curvature $\operatorname{Sec}_g\geq0$ when $d\geq4$. For $d\geq5$, we also construct counterexamples satisfying $\Ric_g\geq0$ and $\nabla_g^2V\geq g$, where $\mu=Z^{-1}e^{-V}\dvol_g$.
% In dimension two, we use inverse mean curvature flow to construct a contracting transport map from the appropriately rescaled round sphere to every closed connected Riemannian surface satisfying the corresponding positive Ricci curvature lower bound. This implies the spectral comparison in dimension two. Together with the recent counterexample in dimension three constructed by Lin, Wang, and Xu
% \cite{lin2026eigenvalues} and our counterexamples in dimensions $d\geq4$, this settles the spherical spectral comparison in every dimension $d\geq2$. Using the same method, we also construct a contracting transport from the uniform probability measure on a hemisphere onto the
% normalized uniform measure on any geodesically convex subset of positive volume, thereby answering affirmatively the remaining case of a conjecture by Beck and Jerison \cite{beck2021friedland} following the work of Fathi, Fradelizi, Gozlan, and Zugmeyer \cite{fathi2023some}.
\end{abstract}


\thanks{}

\maketitle
\section{Introduction} 
\subsection{Background}
In \cite{caffarelli2000monotonicity}, Caffarelli proved the following theorem:
\begin{thm}[cf. Theorem 11 \cite{caffarelli2000monotonicity}]\label{thm:caff}
Let $d\geq 1$, and let $\gamma^d$ denote the standard Gaussian measure on $\R^d$. If $\mu(\ud x)=\gamma^d(\ud x)$ and $\nu(\ud x) = e^{-W(x)} \gamma^d(\ud x)$ are two probability measures where $W$ is a convex function, then there exists a convex function $\phi : \mathbb{R}^d \to \mathbb{R}$ such that $T=\nabla \phi$ is $1$-Lipschitz, and transports $\mu$ to $\nu$, i.e. $T_{\#}\mu =\nu.$
\end{thm}
We can interpret this theorem in the context of $\operatorname{CD}(\rho,\infty)$ spaces for $\rho>0.$ For a smooth weighted manifold $(M,g,e^{-V}\dvol_g)$, the curvature-dimension condition $\CD(\rho,\infty)$ is satisfied when
\begin{align}
    \Ric_g+\nabla_g^2V\geq\rho g.
\end{align}
The standard Gaussian space satisfies this inequality with equality when $\rho=1$. When $W$ is smooth, the target measure in Theorem~\ref{thm:caff} has potential $V(x)=|x|^2/2+W(x)$, up to an additive constant, and hence also satisfies $\CD(1,\infty)$. Thus, Caffarelli's theorem provides a measure-preserving contraction from Gaussian space to a weighted Euclidean space satisfying $\CD(1,\infty)$.

As a consequence, Theorem~\ref{thm:caff} allows one to transfer several functional and geometric inequalities such as the Gaussian Poincar\'e, logarithmic Sobolev, and
isoperimetric inequalities to probability measures
that are more log-concave than $\gamma^d$, with the
same constants (cf. \cite[Section~1.3]{milman2018spectral}). Given these applications as well as several results in comparison geometry, Villani \cite{villani2009optimal} asked if analogs of this theorem would hold on curved spaces.

To this end, Milman in  \cite{milman2018spectral} formulated precise
conjectures extending this observation to the round
sphere and to weighted manifolds. In this work, we
construct counterexamples to these conjectures in
sufficiently high dimensions. We first recall these conjectures below,
\begin{con}[cf. Conjecture $4$ \cite{milman2018spectral}]\label{conj:sphere-transport}
For any $(\bS^d,g,\vol_g)$ satisfying $\operatorname{Ric}_g\geq \rho g$ with $\rho>0$, there exists a map
\begin{align}
        T:(\bS^d,g^\rho_{\can},\vol_{g^\rho_{\can}})
        \to
        (\bS^d,g,\vol_g),
\end{align}
pushing forward $\vol_{g^\rho_{\can}}$ onto $\vol_g$ up to a finite constant and contracting the corresponding metrics. Here $g^\rho_{\can}$ denotes the canonical metric on $\bS^d$ rescaled to have $ \operatorname{Ric}_{g^\rho_{\can}}=\rho g^\rho_{\can}.$
\end{con}
\begin{con}[cf. Conjecture $2^*$ \cite{milman2018spectral}]\label{conj:gaussian-transport}
For any $(\bR^d,g,\mu)$ satisfying $\operatorname{CD}(\rho,\infty)$ with $\rho>0$, there exists a map
\begin{align}
        T:(\bR^d,|\cdot|,\gamma_\rho^d)
        \to
        (\bR^d,g,\mu),
\end{align}
pushing forward $\gamma_\rho^d$ onto $\mu$ up to a finite constant and contracting the corresponding metrics. Here $\gamma_\rho^d(\ud x):= \left( \frac{\rho}{2\pi}\right)^{d/2}e^{-\rho|x|^2/2}\,\ud x$ denotes the rescaled Gaussian measure.
\end{con}
The existence of such contracting transport maps implies a comparison of all eigenvalues. To state this we denote by $\lambda_k(M,g,\mu)$ the $k$-th variational eigenvalue of the nonnegative weighted Laplacian on a smooth weighted manifold $(M,g,\mu)$. When the spectrum is discrete, these are the eigenvalues counted with multiplicity, with the convention $0=\lambda_1\leq\lambda_2\leq\cdots$ on a connected manifold $M$ with finite weighted volume. The contraction principle of Milman in Theorem 1.7 in \cite{milman2018spectral} (cf. Theorem~\ref{thm:contraction}) states that a $1$-Lipschitz map from $(M_1,g_1,\mu_1)$ to $(M_2,g_2,\mu_2)$ pushing forward $\mu_1$ onto $\mu_2$ up to a finite positive constant implies
\begin{align}
        \lambda_k(M_2,g_2,\mu_2)
        \geq
        \lambda_k(M_1,g_1,\mu_1)
        \qquad
        \forall k\geq1.
\end{align}
Therefore, Conjecture~\ref{conj:sphere-transport} implies
\begin{con}[cf. Conjecture $3$ \cite{milman2018spectral}]\label{conj:sphere-spectral}
For any $(\bS^d,g,\vol_g)$ satisfying $\operatorname{Ric}_g\geq \rho g$ with $\rho>0$, we have
\begin{align}
        \lambda_k(\bS^d,g,\vol_g)
        \geq
        \lambda_k(\bS^d,g^\rho_{\can},\vol_{g^\rho_{\can}})
        \qquad
        \forall k\geq 1.
\end{align}
\end{con}
After rescaling the Ricci curvature lower bound, Conjecture
\ref{conj:sphere-spectral} is a subcase of a long-standing
eigenvalue comparison problem stated by Colding and Minicozzi in \cite[p.~261]{colding1998weyl}. In particular, they asked whether every closed Riemannian manifold $(N^d,h)$ satisfying $\operatorname{Ric}_h\geq(d-1)h$ also satisfies
\begin{align}\label{eqn:CM-long-standing-problem}
        \lambda_k(N^d,h)
        \geq
        \lambda_k(\mathbb S^d,g_{\can})
        \qquad
        \forall k\geq1,
\end{align}
where the eigenvalues are counted with multiplicity. They also asked whether equality in \eqref{eqn:CM-long-standing-problem} implies that $(N^d,h)$ is isometric to the round sphere. Colding and Minicozzi observed that this problem is closely related, through the metric cone, to Yau's conjecture asking whether Euclidean space has the maximal dimension of polynomial growth harmonic functions among complete manifolds with nonnegative Ricci curvature, see \cite[Problem 48]{yau1992open}.

Furthermore, Conjectures~\ref{conj:sphere-transport} and~\ref{conj:sphere-spectral} are consistent with the Lichnerowicz theorem, the Bérard--Gallot heat-trace comparison, Weyl asymptotics, and a comparison up to a dimension-independent multiplicative constant for sufficiently large eigenvalues. Similarly, the existence of a contracting transport map in Conjecture \ref{conj:sphere-transport} would imply many fundamental results in comparison geometry, including the Bonnet--Myers diameter bound, Bishop--Gromov volume comparison, the Bakry--\'Emery log-Sobolev inequality, the sharp Sobolev inequality, and the L\'evy--Gromov isoperimetric inequality. We refer to \cite[Sections 1 and 2]{milman2018spectral} and Section 1 in \cite{han-zhu} for more details.

On the other hand, Conjecture~\ref{conj:gaussian-transport} implies
\begin{con}[cf. Conjecture $1^*$ \cite{milman2018spectral}]\label{conj:gaussian-spectral}
For any $(\bR^d,g,\mu)$ satisfying $\operatorname{CD}(\rho,\infty)$ with $\rho>0$, we have
\begin{align}
        \lambda_k(\bR^d,g,\mu)
        \geq
        \lambda_k(\bR^d,|\cdot|,\gamma_\rho^d)
        \qquad
        \forall k\geq 1.
\end{align}
\end{con}
Milman stated the preceding Gaussian conjectures tentatively and proposed the following variant with separate lower bounds on the Ricci tensor and the Hessian of the potential.
\begin{con}[cf. Conjectures $1$ and $2$ \cite{milman2018spectral}]\label{conj:gaussian-combined}
Conjectures~\ref{conj:gaussian-spectral} and
\ref{conj:gaussian-transport} hold when restricted to weighted manifolds diffeomorphic to $\R^d$ and satisfying
\begin{align}
        \Ric_g
        \geq
        0,
        \qquad
        \nabla_g^2V
        \geq
        \rho g.
\end{align}
Here $\mu=Z^{-1}e^{-V}\dvol_g,$ where $V:\R^d\to\R$ is smooth and $Z$ is a normalizing constant.
\end{con}
\subsection{Main Results}
We first show that the spherical conjectures fail
in every $C^\infty$ neighborhood of the unit round
metric in dimensions $d\geq4$.
\begin{mainthm}\label{thm:sphere-main}
Let $d\geq4$. Then there exist an integer $K\geq1$, a number
$\veps_d>0$, and a smooth family of Riemannian metrics
$\{\widehat g_\veps\}_{0\leq\veps<\veps_d}$ on $\bbS^d$ such that
$\widehat g_0=g_\can$ and
\begin{align}\label{eqn:sphere-local-convergence}
        \widehat g_\veps
        \rightarrow
        g_\can
        \qquad
        \text{in }C^\infty
        \text{ as }\veps\to0.
\end{align}
Moreover, for every $0<\veps<\veps_d$, we have
\begin{align}\label{eqn:sphere-local-Ricci}
        \Ric_{\widehat g_\veps}
        \geq
        (d-1)\widehat g_\veps,
\end{align}
but
\begin{align}\label{eqn:index-K}
        \lam_K(\bbS^d,\widehat g_\veps)
        <
        \lam_K(\bbS^d,g_\can).
\end{align}
\end{mainthm}
Thus, Theorem~\ref{thm:sphere-main} gives a negative answer to
Conjecture~\ref{conj:sphere-spectral} and to the question raised by Colding--Minicozzi in \cite{colding1998weyl} even for metrics arbitrarily close to the round metric, while retaining the lower bound $\Ric_g\geq(d-1)g$. Furthermore, by the contraction principle we deduce that Conjecture~\ref{conj:sphere-transport} is also false.

Next, in the $\CD(\rho,\infty)$ setting, we prove the following result.
\begin{mainthm}\label{thm:gaussian-main}
For every $d\geq4$, there exists a smooth complete weighted manifold $(\bR^d,g,\mu)$ satisfying the $\CD(1,\infty)$ condition and $\operatorname{Sec}_g\geq0$ such that
\begin{align}\label{eqn:index-gaussian-K}
\lambda_{d+2}(\bR^d,g,\mu)<2=
\lambda_{d+2}(\bR^d,|\cdot|,
\gamma^d).    
\end{align}
\end{mainthm}
The index $k=d+2$ is the first at which the Gaussian
spectral comparison can fail. This is because, the $\CD(1,\infty)$ condition gives $\lambda_2\geq1$ (cf. Prop 4.8.1 in \cite{bakry2014analysis}) while 
\begin{align}
    \lambda_j(\R^d,|\cdot|,\gamma^d)=1,
    \quad 2\leq j\leq d+1.
\end{align}
Thus, Theorem~\ref{thm:gaussian-main} disproves
Conjecture~\ref{conj:gaussian-spectral} even under the
additional assumption $\operatorname{Sec}_g\geq0$.
The contraction principle therefore also rules out
Conjecture~\ref{conj:gaussian-transport}.

We next show that imposing the separate bounds $\Ric_g\geq0$ and $\nabla_g^2V\geq g$ does not imply the spectral comparison or the existence of a contracting transport.
\begin{mainthm}\label{thm:gaussian-split-main}
For every $d\geq5$, there exist a smooth complete weighted manifold $(\bR^d,g,\mu)$ and an integer $K\geq1$ with the following properties. The measure $\mu$ has the form
\begin{align}
        \mu
        =
        Z^{-1}e^{-V}\dvol_g,
        \qquad
        Z
        :=
        \int_{\bR^d}e^{-V}\dvol_g
        <
        \infty,
\end{align}
where $V:\bR^d\to\bR$ is smooth. Moreover,
$(\bR^d,g,\mu)$ satisfies the $\CD(1,\infty)$ condition,
\begin{align}
        \Ric_g
        \geq
        0,
        \qquad
        \nabla_g^2V
        \geq
        g,
\end{align}
but
\begin{align}\label{eqn:index-gaussian-split-K}
        \lambda_K(\bR^d,g,\mu)
        <
        \lambda_K(\bR^d,|\cdot|,\gamma^d).
\end{align}
\end{mainthm}
Thus, Conjecture~\ref{conj:gaussian-combined} is also false in general.

\subsection{Related Work}
We remark that while these generalizations of Caffarelli's contraction theorem are false, there are several positive generalizations, as well as new proofs of this theorem. These include results in infinite-dimensional spaces \cite{feyel2004monge,mikulincer2024brownian}, results for $1/d$-concave densities \cite{carlier2024optimal}, and results in curved settings \cite{fathi2024transportation,lopez2025bakry}. Kim and Milman \cite{kim2012generalization} constructed Lipschitz transports using the heat flow, while entropic optimal transport has led to further proofs and extensions \cite{chewi2023entropic,fathi2020proof}. 

Conjecture~\ref{conj:sphere-transport} has been previously studied in dimension $d=2$ \cite{serres2025contractivetransportmapsmathbbs2} and in dimensions $d\geq 2$ in \cite{ge2025generalization} where the authors constructed contracting transport maps from the round sphere to a nearly spherical manifold. The metrics on the sphere for which a contracting map exists in \cite{ge2025generalization} require an additional smallness condition relative to a round sphere of radius $\rho<1$ which is not satisfied by the counterexamples constructed in the proof of Theorem~\ref{thm:sphere-main}.

Following the first version of this work, Lin, Wang, and Xu \cite{lin2026eigenvalues} proved Conjecture~\ref{conj:sphere-spectral} when $d=2$ and constructed a counterexample when $d=3$. The author, together with Han and Zhu \cite{han-zhu}, proved Conjecture~\ref{conj:sphere-transport} for smooth two-spheres using inverse mean curvature flow. Thus, together with Theorem~\ref{thm:sphere-main}, these results show that Conjectures~\ref{conj:sphere-transport} and \ref{conj:sphere-spectral} hold in dimension two and fail in every dimension $d\geq3$.

Recently, Dudarov and Mikulincer \cite{dudarov2026geometric} constructed a complete weighted manifold on $\R^2$ satisfying $\CD(1/2,\infty)$ that admits no Lipschitz transport from a Gaussian space. They also showed that its eigenvalues, divided by the corresponding eigenvalues of the Gaussian space, tend to zero. Their example has negative Ricci curvature somewhere (cf. Remark 1 in \cite{dudarov2026geometric}). In contrast, our counterexamples in Theorem~\ref{thm:gaussian-main} have $\operatorname{Sec}_g\geq0$ and those in Theorem~\ref{thm:gaussian-split-main} satisfy $\operatorname{Ric}_g\geq 0$ and $\nabla_g^2V\geq g$ separately.

\subsection{Proof Sketch}
We briefly explain the main ideas involved in proving our main results.

\underline{\textit{Proof Sketch of Theorem~\ref{thm:sphere-main}}}: To construct the counterexamples in Theorem~\ref{thm:sphere-main},
we need to find metrics that break the spectral comparison
inequality while maintaining the Ricci curvature lower
bound. To this end, we will perturb the standard round
metric, estimate the Rayleigh quotients of suitable test functions from above and the Ricci curvature from below, and then rescale to restore the curvature bound.

To motivate the construction of the metric, consider first
the standard metric on a circle $\bbS^1$, denoted by
$h_0=\ud\theta^2$, and its perturbation
$h_\veps=(1+\veps)\ud\theta^2$.
For every nonzero test function $u\in W^{1,2}(\bbS^1)$,
the Rayleigh quotient satisfies
\begin{align}
\cR_{h_\veps}(u)
&=\frac{(1+\veps)^{-1/2}
\int_0^{2\pi}|u'|^2\,\ud\theta}
{(1+\veps)^{1/2}
\int_0^{2\pi}|u|^2\,\ud\theta}
=\frac{\cR_{h_0}(u)}{1+\veps}.
\end{align}
Thus, scaling the metric by $1+\veps$ multiplies the
Rayleigh quotient by $(1+\veps)^{-1}$. 

Now let $d\geq4$ and write
$\bbS^d=\{(z,y)\in\bbC\times\R^{d-1}:|z|^2+|y|^2=1\}$.
Away from $\{z=0\}\cup\{y=0\}$, we use the coordinates
$(z,y)=(\sin t\,e^{i\theta},\cos t\,\omega)$, where
$t\in(0,\pi/2)$, $\theta\in\bbS^1$, and
$\omega\in\bbS^{d-2}$. The round metric takes the form
\begin{align}
g_0=\ud t^2+\sin^2t\,\ud\theta^2
+\cos^2t\,g_{\bbS^{d-2}}.
\end{align}
To determine how to perturb $g_0$, consider the
spherical harmonics
\begin{align}
f_k=\Re(z^k)=\sin^kt\cos(k\theta)
\end{align}
whose derivatives satisfy
\begin{align}
\frac{\int_{\bbS^d}\sin^{-2}t\,|\partial_\theta f_k|^2
\,\dvol_{g_0}}
{\int_{\bbS^d}f_k^2\,\dvol_{g_0}}
=\frac{k(2k+d-1)}2,\quad \frac{\int_{\bbS^d}|\partial_t f_k|^2\,\dvol_{g_0}}
{\int_{\bbS^d}f_k^2\,\dvol_{g_0}}
=\frac{k(d-1)}2,
\end{align}
while the derivatives in the $\omega$ directions vanish. Thus, the angular contribution to the Rayleigh quotient is of order $k^2$, whereas the contribution from the $t$-derivative is of order $k$. This suggests increasing the metric in the angular direction to decrease the
leading contribution, as in the example on $\bbS^1$. We therefore consider the perturbation
\begin{align}
g_\veps
=\ud t^2+\sin^2t(1+\veps\sin^4t)\ud\theta^2
+\cos^2t\,g_{\bbS^{d-2}}.
\end{align}
The choice $\sin^4t$ is not unique but it allows explicit estimates of both
the Rayleigh quotient and the Ricci tensor. In particular, writing $q=\sin^2t$, the first-order expansion of the Rayleigh quotient involves only integrals of powers of $q$. The estimates below show that the resulting spectral decrease is larger than the increase caused by rescaling the metric to restore the Ricci lower bound
\begin{align}
\frac{\cR_{g_\veps}(f_k)}{k(k+d-1)}
=1-\beta_{d,k}\veps+O_{d,k}(\veps^2),
\quad \Ric_{g_\veps}
\geq\left(d-1-\frac{25}{12}\veps\right)g_\veps,
\end{align}
where $\lim_{k\to\infty}\beta_{d,k}=1$. Set
\begin{align}
a_d:=\frac{25}{12(d-1)},
\qquad
\widehat g_\veps:=(1-a_d\veps)g_\veps.
\end{align}
Then, $\Ric_{\widehat g_\veps} =\Ric_{g_\veps} \geq(d-1)(1-a_d\veps)g_\veps =(d-1)\widehat g_\veps.$ On the other hand, 
\begin{align}
\frac{\cR_{\widehat g_\veps}(f_k)}{k(k+d-1)}
&=\frac{1-\beta_{d,k}\veps+O_{d,k}(\veps^2)}
{1-a_d\veps}=1-(\beta_{d,k}-a_d)\veps+O_{d,k}(\veps^2).
\end{align}
Since $a_d<1$ for $d\geq4$ and $\lim_{k\to \infty}\beta_{d,k} = 1$, we first
fix $k$ sufficiently large that $\beta_{d,k}>a_d$. 
For this fixed $k\geq 1$, taking $\veps>0$ sufficiently small gives
\begin{align}
\cR_{\widehat g_\veps}(f_k)
\leq
\left(1-\frac{\beta_{d,k}-a_d}{2}\veps\right)k(k+d-1)
<k(k+d-1).
\end{align}
To obtain the spectral inequality, let $E_{k-1}$ be the
space of spherical harmonics of degrees $0,\ldots,k-1$,
and set
\begin{align}
F_k:=E_{k-1}\oplus\operatorname{span}\{f_k\},
\qquad K:=\dim F_k.
\end{align}
Rotational symmetry makes $f_k$ orthogonal to $E_{k-1}$
for both the $L^2$ inner product and the Dirichlet form
of $\widehat g_\veps$. Moreover,
\begin{align}
\sup_{0\neq h\in E_{k-1}}\cR_{g_0}(h)
=(k-1)(k+d-2)<k(k+d-1).
\end{align}
By continuity on this finite-dimensional space, the strict
upper bound $k(k+d-1)$ persists for $\widehat g_\veps$
when $\veps$ is sufficiently small. Consequently,
the min--max principle gives
\begin{align}
\lam_K(\bbS^d,\widehat g_\veps)
&\leq\sup_{0\neq u\in F_k}\cR_{\widehat g_\veps}(u)\\
&=\max\left\{
\sup_{0\neq h\in E_{k-1}}\cR_{\widehat g_\veps}(h),
\cR_{\widehat g_\veps}(f_k)\right\}\\
&<k(k+d-1)=\lam_K(\bbS^d,g_0).
\end{align}
The explicit formula for $\widehat g_\veps$ also shows
that $\widehat g_\veps\to g_0$ smoothly as $\veps\to0$.

\underline{\textit{Proof Sketch of Theorem~\ref{thm:gaussian-main}}}: For Theorem~\ref{thm:gaussian-main}, let $d\geq4$.
We need a complete metric on $\R^d$ and a potential $V$
such that $\operatorname{Sec}_g\geq0$ and $\Ric_g+\nabla_g^2V\geq g$,
while $\lam_{d+2}(\R^d,g,\mu)<2$. To motivate the
construction, consider a cylinder with metric
$\ud\tau^2+\delta^2g_{\bbS^{d-1}}$ and potential
$V(\tau)=\tau^2/2$. Its Ricci curvature along $\partial_\tau$ is zero,
while its tangential Ricci curvature satisfies,
\begin{align}
\Ric_g(X,X)=\frac{d-2}{\delta^2},
\quad |X|_g=1,\quad X\perp\partial_\tau.
\end{align}
Furthermore, 
\begin{align}
\lam_2(\bbS^{d-1}(\delta))=\frac{d-1}{\delta^2}.
\end{align}
Since $\nabla_g^2V=\ud\tau^2$, we choose
$\delta^2=d-2$. This is the largest radius allowed by
the curvature bound, and implies that the eigenvalue $\lam_2(\bbS^{d-1}(\delta))<2$ when $d\geq4$.

To close the cylinder at the origin, we first attach
a round hemisphere of radius $\delta$. Set
$r_0=\pi\delta/2$ and consider the comparison model
\begin{align}\label{eqn:V_0}
g_0&=\ud r^2+\rho_0(r)^2g_{\bbS^{d-1}},
\qquad V_0(r)=\frac12(r-r_0)_+^2,\\
\rho_0(r)&=
\begin{cases}
\delta\sin(r/\delta),&0\leq r\leq r_0,\\
\delta,&r\geq r_0.
\end{cases}
\end{align}
We will smooth $\rho_0$ and $V_0$ later near $r=r_0$. Write $\nu_0=e^{-V_0}\dvol_{g_0}$ and let
\begin{align}
\cR_{g_0}(u)
:=\frac{\int_{\R^d}|\nabla u|_{g_0}^2\,\ud\nu_0}
{\int_{\R^d}u^2\,\ud\nu_0},
\end{align}
for a nonzero $u\in W^{1,2}(\R^d,g_0,\nu_0)$. For $\omega=(\omega_1,\ldots,\omega_d)\in\bbS^{d-1}\subset \R^d$,
choose the test functions
\begin{align}
U_i(r,\omega)=\frac{\rho_0(r)}{\delta}\omega_i,
\quad b(r)&=
\begin{cases}
\cos(r/\delta),&0\leq r\leq r_0,\\
-\dfrac{\delta}{d}(r-r_0),&r\geq r_0.
\end{cases}
\end{align}
The function $b$ is continuous and piecewise smooth,
with bounded derivative and at most linear growth.
The Gaussian weight on the end therefore gives
$b\in W^{1,2}(\R^d,g_0,\nu_0)$.
The coefficient $\delta/d$ makes it mean zero with
respect to $e^{-V_0}\dvol_{g_0}$. On the cap, when $r\leq r_0$, $U_i$ and $b$
are coordinate functions on the round hemisphere. Since $\partial_rU_i=0$ and $b=0$ at $r=r_0$, integration by parts gives the following estimates on the Rayleigh quotients 
\begin{align}
\cR_{\leq r_0}(U_i)
=\cR_{\leq r_0}(b)
=\frac{d}{d-2}\leq2,
\quad
\cR_{\geq r_0}(U_i)=\frac{d-1}{d-2}<2,
\quad
\cR_{\geq r_0}(b)=1.
\end{align}
Here the subscripts mean that the integrals in the Rayleigh quotients are restricted to the indicated region and are computed with respect to the metric $g_0.$   Therefore, 
\begin{align}
\cR_{g_0}(U_i)<2,\qquad \cR_{g_0}(b)<2.
\end{align}
Note that these global inequalities are strict because each test function has positive squared $L^2$ norm on
the cylindrical end, where its Rayleigh quotient is strictly less than $2$. Set $F_0:=\operatorname{span}\{1,b,U_1,\ldots,U_d\}.$ Rotational symmetry and the mean-zero property of $b$ imply that these functions are pairwise orthogonal for both the $L^2(\nu_0)$ inner product and the Dirichlet form. Consequently, $\dim F_0=d+2$ and
\begin{align}
    \sup_{0\neq u\in F_0}\cR_{g_0}(u)
    =
    \max\left\{\cR_{g_0}(b),
    \max_{1\leq i\leq d}\cR_{g_0}(U_i)\right\}
    <2.
\end{align}
In our actual argument, we will smooth $\rho_0$ and $V_0$ in Section~\ref{sec:proof-of-gaussian-main} while preserving the required curvature bounds and the strict spectral inequality above on a corresponding space of test functions. The min--max principle then gives us the desired conclusion.

\underline{\textit{Proof Sketch of Theorem~\ref{thm:gaussian-split-main}}}: The construction in Theorem~\ref{thm:gaussian-main} does not give the separate Hessian bound required in
Theorem~\ref{thm:gaussian-split-main}, since a radial
potential has zero tangential Hessian on the cylindrical
end. We therefore consider a metric cone. Let
$n=d-1\geq4$. For a metric $\bar g$ on $\bbS^n$, set
\begin{align}
g_C=\ud r^2+s^2r^2\bar g,
\qquad V_C=\frac{\alpha r^2}{2},
\end{align}
where $\alpha>0$ and $0<s<1$ will be chosen below.
Away from the vertex,
\begin{align}
\nabla_{g_C}^2V_C=\alpha g_C,
\qquad
\Ric_{g_C}\geq0
\quad\text{if}\quad
\Ric_{\bar g}\geq(n-1)s^2\bar g.
\end{align}
Thus the quadratic potential supplies the Hessian bound
in every direction when $\alpha\geq1$. Our goal is to
choose $\bar g,s,\alpha$ so that the cone also admits
a finite-dimensional space of test functions whose Rayleigh quotients
lie below the corresponding Gaussian eigenvalue, and
then smooth the vertex of the cone while preserving these inequalities.

To construct $\bar g$, we again lengthen the circle
direction on the sphere. We now contract the other
$n-1$ directions by a compensating factor, so that the
volume form remains unchanged. In the coordinates
$z=\sin t\,e^{i\theta}$ used earlier, consider the family of metrics
\begin{align}
g_\tau
&=e^{-2\tau\veps\sin^2t/(n-1)}
\bigl(\ud t^2+\cos^2t\,g_{\bbS^{n-2}}\bigr)
+e^{2\tau\veps\sin^2t}\sin^2t\,\ud\theta^2, \quad \dvol_{g_\tau}=\dvol_{g_{\bbS^n}},
\text{ for } 0\leq\tau\leq1
\end{align}
and set $\bar g:=g_1.$ Keeping the volume form fixed for all $\ta\in [0,1]$ will allow us to
compare polynomial test functions with those of
Gaussian space. It also implies that $\operatorname{tr}(g_\tau^{-1}\partial_\tau g_\tau)=0$,
which is used in smoothing the metric $g_C$ constructed using $\bar g.$ For sufficiently small $\veps>0$, the Ricci curvature and Rayleigh quotient satisfy the following estimates
\begin{align}
&\Ric_{g_\tau}\geq(n-1)\rho^2g_\tau
\text{ for every }\tau\in[0,1], \text{ where }\rho^2=1-\frac{C_n}{n-1}\veps,\quad  C_n=n-1+\frac{2n}{n-1},\\
&\frac{\cR_{\bar g}(f_k)}{k(k+n-1)}
\leq1-\kappa_{n,k}\veps+O_{n,k}(\veps^2),
\text{ where } f_k=\Re(z^k) \text{ and } \lim_{k\to \infty} \kappa_{n,k}=2.
\end{align}
Since $C_n/(n-1)<2$ for $n\geq4$, we first choose
an integer $k\geq2$ large enough that
$\kappa_{n,k}>C_n/(n-1)$, and then choose
$\veps>0$ sufficiently small. We can then choose
$0<s<\rho<1$, with $s=1+O_{n,k}(\veps)$, such that
\begin{align}
\cR_{\bar g}(f_k)<s^2k(k+n-1).
\end{align}
This comparison explains the restriction $d\geq5$.

We now construct the space of test functions on the cone. Let
$\cR_C$ denote the Rayleigh quotient with respect to
$e^{-V_C}\dvol_{g_C}$, let $\mathcal P_{<k}$ be the
space of polynomials on $\R^d$ of degree less than $k$,
and set $P_k(r,\omega)=r^kf_k(\omega)$ for
$x=r\omega$ with $\omega\in\bbS^n$. To estimate the Dirichlet energy with respect to $g_C$, it is helpful to observe that for every $p\in\mathcal P_{<k}$, we have
\begin{align}
|\nabla_{g_C}p|_{g_C}^2
&\leq A|\nabla p|^2,
\qquad
A:=\max\left\{1,s^{-2}e^{2\veps/(n-1)}\right\}
=1+O_{n,k}(\veps),
\end{align}
where $\nabla$ denotes the Euclidean gradient. Moreover, $e^{-V_C}\dvol_{g_C} =s^n e^{-\alpha|x|^2/2}\,\ud x$ since $\dvol_{\bar g}=\dvol_{g_{\bbS^n}}$. The constant factor $s^n$ cancels in the Rayleigh quotient when one divides the Dirichlet energy by the squared $L^2$ norm of the test function. Thus, we can estimate the Rayleigh quotient as follows
\begin{align}
\sup_{0\neq p\in\mathcal P_{<k}}\cR_C(p)
&\leq\alpha A(k-1),\\
\cR_C(P_k) &= \alpha B_k, \text{ where } B_k:= \frac{k^2+s^{-2}\cR_{\bar g}(f_k)} {n+2k-1} <k,
\end{align}
where the estimate on $B_k$ uses $\cR_{\bar g}(f_k)<s^2k(k+n-1)$. Taking $\veps>0$ sufficiently small in the preceding construction, we also have $A(k-1)<k$. We may therefore
choose
\begin{align}
1<\alpha<\min\left\{
\frac{k}{A(k-1)},\frac{k}{B_k}\right\}.
\end{align}
We then show that $P_k$ and the space $\calP_{<k}$ are orthogonal for both the
weighted $L^2$ inner product and the Dirichlet form which implies that 
\begin{align}
\sup_{0\neq p\in E_C}\cR_C(p)
<k=\lam_K(\R^d,|\cdot|,\gamma^d),
\end{align}
where $E_C:=\mathcal P_{<k}\oplus\operatorname{span}\{P_k\},$ and $K:=\dim E_C=\binom{d+k-1}{d}+1.$ It remains to smooth the vertex while preserving
the curvature bounds and the strict spectral
inequality above. Here, Lemma~\ref{lem:gaussian-split-smoothing}
constructs, for sufficiently small $\eta>0$, a
smooth complete metric $g_\eta$ and a smooth
potential $V_\eta$ on $\R^d$ satisfying
\begin{align}
\Ric_{g_\eta}\geq0,
\qquad
\nabla_{g_\eta}^2V_\eta\geq g_\eta.
\end{align}
The strict inequalities $s<\rho$ and $\alpha>1$ are used in this construction to preserve the Ricci and Hessian lower bounds during the smoothing process.  Define
\begin{align}
\mu_\eta:=Z_\eta^{-1}e^{-V_\eta}\dvol_{g_\eta},
\qquad
Z_\eta:=\int_{\R^d}e^{-V_\eta}\dvol_{g_\eta}<\infty,
\end{align}
and let $\cR_\eta$ denote the Rayleigh quotient
on $(\R^d,g_\eta,\mu_\eta)$. The same lemma (cf. Lemma~\ref{lem:gaussian-split-smoothing}) also constructs a $K$-dimensional test space $E_\eta$ whose Rayleigh quotients remain strictly below $k$.
The min--max principle therefore gives
\begin{align}
\lam_K(\R^d,g_\eta,\mu_\eta)
\leq\sup_{0\neq u\in E_\eta}\cR_\eta(u)
<k=\lam_K(\R^d,|\cdot|,\gamma^d),
\end{align}
as desired.

\subsection{Additional Remarks}
Finally, we remark that the counterexamples in this work lead to some natural questions:
\begin{itemize}
    \item Are Conjectures~\ref{conj:sphere-transport} and \ref{conj:sphere-spectral} true under the stronger assumption $\operatorname{Sec}_g\geq\rho/(d-1)$? Equivalently, after rescaling, does $\operatorname{Sec}_g\geq1$ imply the spectral comparison with the unit round sphere and the existence of a contracting transport from it? The metrics in Theorem~\ref{thm:sphere-main} have positive sectional curvature for sufficiently small $\veps>0$, but do not satisfy this uniform lower bound.

    \item Does Conjecture~\ref{conj:gaussian-combined} hold in dimensions $d=2,3,4$? The three-dimensional counterexamples for Conjecture~\ref{conj:sphere-spectral} in \cite{lin2026eigenvalues} suggest asking whether the methods used in the proof of Theorem~\ref{thm:gaussian-split-main} could be adapted to give a counterexample to Conjecture~\ref{conj:gaussian-combined} when $d=4$.
    
    \item It would also be interesting to find an analogue of
Caffarelli's contraction theorem on hyperbolic space,
using the heat kernel measure in place of the Gaussian
measure. This leads to the following question. Let $(\mathbb H^d,g)$ be the standard hyperbolic space with constant negative sectional curvature, fix $o\in\mathbb H^d$ and $t>0$, and set $\mu_t:=p_t(o,\cdot)\dvol_g,$ where $p_t$ is the heat kernel. If $W:\mathbb H^d\to\R$ is smooth and geodesically convex,
and $Z:=\int_{\mathbb H^d}e^{-W}\,\ud\mu_t$ is finite,
does there exist a contracting map
\begin{align}
    T:(\mathbb H^d,g,\mu_t)
    \longrightarrow
    (\mathbb H^d,g,Z^{-1}e^{-W}\mu_t)
\end{align}
pushing forward the source measure onto the target measure? Note that the existence of such a map in $\R^d$ follows from Caffarelli's Theorem~\ref{thm:caff} after rescaling.
\end{itemize}

\subsection{Organization}
The paper is organized as follows. In Section~2, we recall the variational definition of eigenvalues and the contraction principle. Sections~3, 4, and 5 contain the proofs of Theorems~\ref{thm:sphere-main}, \ref{thm:gaussian-main}, and \ref{thm:gaussian-split-main}, respectively.

% Besides Caffarelli's contraction theorem, our interest in contracting
% transport maps on curved spaces was also motivated by a conjecture of Yau, which states that if
% \(\Sigma^2\subset\mathbb S^{3}\) is a closed embedded minimal hypersurface,
% then its first nonzero Laplace--Beltrami eigenvalue is \(2\) (cf. \cite{yau1982seminar}). Using the convention $0=\lambda_1(\Sigma)\leq\lambda_2(\Sigma)\leq\cdots,$ one can show that this is equivalent to $\lam_2(\Sigma)\geq 2$ since the coordinate functions restricted to $\Sigma$ are eigenfunctions with eigenvalue $2$.
% For \(n=2\), a contracting measure-preserving map
% \(T:\mathbb S^2\to\Sigma\) would imply, by Theorem~\ref{thm:contraction}, that
% \begin{align}
%         \lambda_k(\Sigma)\geq\lambda_k(\mathbb S^2)
%         \qquad\forall k\geq1,
% \end{align}
% which in particular would imply the desired estimate when $k=2.$ However, this stronger conclusion is false in general. For instance, for the Clifford
% torus
% \begin{align}
%         \Sigma
%         =
%         \mathbb S^1\!\left(\frac1{\sqrt2}\right)
%         \times
%         \mathbb S^1\!\left(\frac1{\sqrt2}\right)
%         \subset\mathbb S^3,    
% \end{align}
% we have (cf. Remark 2 $\mathrm{(ii)}$ in \cite{muto1984homogeneous})
% \begin{align}
%  \lambda_5(\Sigma)=2<6=\lambda_5(\mathbb S^2).    
% \end{align}
% The fact that a product structure as above can give rise to high multiplicity of small eigenvalues is the key observation behind the construction of the counterexamples in this work.
\subsection{Acknowledgements}
The author would like to thank his advisor, Tobias Colding, for his support over the years and for many insightful conversations. The author is also grateful to David Jerison for introducing him to the Caffarelli contraction theorem and for several illuminating discussions about transport maps. Finally, the author would also like to thank Max Fathi, Promit Ghosal, Dan Mikulincer, and Yair Shenfeld for many helpful discussions. This work was supported in part by a Simons Dissertation Fellowship and NSF DMS Grant 2405393.

\section{Preliminaries}
A weighted manifold is a triple \((M^d,g,\mu)\), where \((M,g)\) is a complete
smooth Riemannian manifold,
\[
        \ud\mu=e^{-W}\dvol_g,
\]
and \(W\) is smooth. The weighted Laplacian is defined as:
\begin{align}
\Delta_{g,\mu}f= \Delta_g f-\langle\nabla_g W,\nabla_g f\rangle_g.    
\end{align}
We use the nonnegative operator $-\Delta_{g,\mu}$ and for $k\geq1$, define the $k$-th variational eigenvalue by
\begin{align}
        \lambda_k(M,g,\mu)
        :=
        \inf_{\substack{
        E\subset W^{1,2}(M,g,\mu)\\
        \dim E=k
        }}
        \sup_{0\neq f\in E}
        \frac{
        \int_M|\nabla f|_g^2\,\ud\mu
        }{
        \int_Mf^2\,\ud\mu
        }.
\end{align}
When the corresponding weighted Laplacian has compact resolvent, finite volume $\mu(M)<\infty$ and $M$ is connected then these variational eigenvalues are the usual eigenvalues, counted with multiplicity. We can write these eigenvalues in non-decreasing order as
\[
0=\lambda_1(M,g,\mu)\leq \lambda_2(M,g,
\mu)\leq\cdots,
\]
including multiplicity. The Bakry--\'Emery curvature tensor is
\[
\Ric_{g,
\mu}:=\Ric_g+\nabla_g^2W.
\]
From Theorem 0.12 \cite{lott2009ricci}, a weighted manifold $(M^d,g,\mu)$ is said to satisfy curvature-dimension condition $\CD(\rho,
\infty)$ for $\rho \in \R$ if 
\begin{align}
\Ric_g+\nabla_g^2W\geq \rho g.    
\end{align}
Let $\gamma^d(\ud x)=(2\pi)^{-d/2}e^{-|x|^2/2}\,\ud x$ denote the standard Gaussian measure on $\bR^d$, with the weighted Laplacian being the Ornstein--Uhlenbeck operator whose first few eigenvalues satisfy (cf. Section 2.1 in \cite{milman2018spectral})
\begin{align}\label{eqn:gaussian-spectrum}
\lambda_1(\R^d,|\cdot|,\gamma^d)=0,
\qquad
\lambda_2(\R^d,|\cdot|,\gamma^d)=\cdots=\lambda_{d+1}(\R^d,|\cdot|,\gamma^d)=1,
\qquad
\lambda_{d+2}(\R^d,|\cdot|,\gamma^d)=2.    
\end{align}
Note that $(\R^d,|\cdot|,\gamma^d)$ is $\CD(1,\infty).$  Next, we recall the following theorem due to Milman \cite{milman2018spectral} that relates spectral comparison between weighted manifolds to the Lipschitz transport maps.
\begin{thm}[Contraction Principle]\label{thm:contraction}
Let
\[
    T : (M_1, g_1, \mu_1) \to (M_2, g_2, \mu_2)
\]
denote an $L$-Lipschitz map between two complete weighted manifolds
pushing forward $\mu_1$ onto $\mu_2$ up to a finite constant. Then
\[
    \lambda_k(M_2, g_2, \mu_2)
    \geq
    \frac{1}{L^2}\lambda_k(M_1, g_1, \mu_1)
    \qquad \forall k \geq 1.
\]
\end{thm}
In particular, a $1$-Lipschitz transport map from a model space to a target space implies that the spectrum of the target space dominates that of the source.

\section{Proof of Theorem~\ref{thm:sphere-main}}\label{sec:sphere-thm-main}
Let $d\geq 4$. We denote the round sphere $\bbS^d=\{(z,y)\in \mathbb C\times \mathbb R^{d-1}: |z|^2+|y|^2=1\}$ and will write \(z=x_1+\mathrm i x_2\). Consider the map
\begin{align}
    F:(0,\pi/2)\times \bbS^1\times \bbS^{d-2} \to \mathbb{S}^d \subset \mathbb{C}\times \bbR^{d-1}
\end{align}
such that 
\begin{align}
    F(t,\theta,\omega) = (e^{i\theta}\sin t, (\cos t) \omega).
\end{align}
Then $F$ maps into the sphere $\mathbb{S}^d$ and is a diffeomorphism onto the open set $U:= \{(z,y)\in\mathbb S^d\subset\mathbb C\times\mathbb R^{d-1}:z\neq0,\ y\neq0\}.$ Let $g_0$ denote the induced Euclidean metric on $\mathbb{S}^d$, then in local coordinates on $U$, we can write the metric $g_0$ using the pull-back metric $F^*g_0$ as follows:
\begin{align}\label{eqn:expression-for-g_can-Sn}
    g_0  &= |\ud z|^2 + |\ud y|^2\\
    &= |e^{i\theta} (\cos t)  \ud t+ ie^{i\theta} (\sin t)  \ud \theta|^2 + |-(\sin t) \omega \ud t + (\cos t) \ud \omega|^2\\
    &= \cos^2 t \ud t^2 + \sin^2 t \ud \theta^2 + \sin^2 t \ud t^2 + (\cos^2 t) g_{\mathbb{S}^{d-2}}\\
    &= \ud t^2+\sin^2t\,\ud \theta^2+\cos^2t\,g_{\mathbb S^{d-2}},
\end{align}
where in the above expression we used the orthogonality of $e^{i\theta}$ and $ie^{i\theta}$ with respect to the real Euclidean inner product in $\C\simeq \R^2$ and the fact $|\omega|^2=1$ which after differentiating implies that $\la \omega, \ud \omega\ra=0$.

The idea is to perturb this metric slightly in the \(\bbS^1\)-direction. As
expected, this lowers the Ricci curvature lower bound from \(d-1\) to some value $\rho$, but it will also drop the eigenvalues of the perturbed manifold below that of the sphere with the round metric and $\Ric=\rho g.$ Note that, $g_0=g_{\can}$ where $g_{\can}$ denotes the canonical round metric on the unit sphere $\bbS^d$ and $\Ric_{g_0}=(d-1)g_0.$

Before proceeding with the proof of Theorem~\ref{thm:sphere-main}, we record a
useful lemma (cf. Lemma 2.3 in \cite{karaca2018gradient}) which we will use to
compute the Ricci curvature of warped product metrics.
\begin{lem}[Ricci tensor of a multiple warped product]\label{lem:multiply-warped-ricci}
Let \((B,g_B)\) and \((F_i,g_i)\), \(1\leq i\leq m\), be Riemannian
manifolds, and let $b_i:B\to (0,\infty)$ be smooth positive functions. Let $M=B\times_{b_1}F_1\times\cdots\times_{b_m}F_m$ be the multiply warped product with metric
\[
        g
        =
        g_B+\sum_{i=1}^m b_i^2 g_i.
\]
Write $s_i=\dim F_i.$ Let \(X,Y\) be vector fields tangent to the base \(B\), and let \(V_i,W_i\)
be vector fields tangent to the fiber \(F_i\), all lifted to \(M\). Then the
Ricci tensor of \(g\) is given as follows:
\begin{enumerate}
    \item $\operatorname{Ric}_g(X,Y)
        =
        \operatorname{Ric}_B(X,Y)
        -
        \sum_{i=1}^m
        \frac{s_i}{b_i}\operatorname{Hess}_B b_i(X,Y).$
    \item $\operatorname{Ric}_g(X,V_i)=0.$
    \item if
\(i\neq j\), then
$\operatorname{Ric}_g(V_i,V_j)=0.$
    \item for each $1\leq i\leq m$ we have
    \begin{align}
      \operatorname{Ric}_g(V_i,W_i)
        &=
        \operatorname{Ric}_{F_i}(V_i,W_i)
        \\
        &\quad
        -
        g(V_i,W_i)
        \left(
        \frac{\Delta_B b_i}{b_i}
        +
        (s_i-1)\frac{|\nabla_B b_i|^2}{b_i^2}
        +
        \sum_{\substack{k=1\\ k\neq i}}^m
        s_k
        \frac{\langle\nabla_B b_i,\nabla_B b_k\rangle_B}{b_i b_k}
        \right).  
    \end{align}
    
\end{enumerate}
Here all gradients, Hessians, Laplacians, and inner products appearing on
the right-hand side are computed with respect to the base metric \(g_B\).
\end{lem}

\begin{proof}[Proof of Theorem~\ref{thm:sphere-main}]
Let \(d\geq 4\). Our proof proceeds in four main steps. 
\paragraph{$\mathrm{(i)}$ \textit{Construction of the metric}} For \(0<\varepsilon\ll 1\), define a metric on the coordinate chart $U$ as 
\begin{equation}\label{eq:geps-coordinate}
g_\varepsilon=
        \ud t^2 +
        \sin^2t\bigl(1+\varepsilon\sin^4t\bigr)\ud \theta^2 + \cos^2t\,g_{\mathbb S^{d-2}}.
\end{equation}
This is a smooth, well-defined metric on $\bbS^d$ as can be seen by rewriting 
\begin{align}\label{eq:geps-global}
g_\varepsilon= g_0+\varepsilon |z|^2\alpha\otimes\alpha, 
\end{align}
where $\alpha=x_1\,dx_2-x_2\,dx_1=\sin^2 t \ud \theta$ and we used the fact that $|z|^2=\sin^2 t.$
\paragraph{$\mathrm{(ii)}$ \textit{Ricci curvature lower bound}} We now estimate its Ricci tensor. Set 
\begin{align}
a_{\veps}(t):= (\sin t)\sqrt{1+\veps \sin^4 t},\quad c(t):=\cos t 
\end{align}
Then, working in a local orthonormal frame with basis vectors $e_t=\partial_t$, $e_\theta=\frac{1}{a_\veps(t)} \partial_\theta$ and $e_j=\frac{1}{\cos t} v_j$ where $\{v_j\}_{j=1}^{d-2}$ is a local $g_{\bbS^{d-2}}$-orthonormal frame on the $\bbS^{d-2}$-factor, we can denote 
\begin{align}
R_t=\Ric_{g_\veps}(e_t,e_t),\quad R_\theta=\Ric_{g_{\veps}}(e_\theta,e_\theta),\quad R_{S}=\Ric_{g_\veps}(e_j,e_j) \text{ for }1\leq j\leq d-2.
\end{align}
Note that the last term is independent of any $1\leq j\leq d-2$ since the Ricci curvature of $\bbS^{d-2}$ is a constant multiple of the $g_{\bbS^{d-2}}.$ Using Lemma \ref{lem:multiply-warped-ricci} we can compute each term as follows. First,
\begin{align}
        -\frac{a_\veps''}{a_\veps}
        =
        1-\veps\left(\frac{\sin^2 t(10-12\sin^2 t)+2\veps \sin^6 t (3-4\sin^2 t)}{(1+\veps \sin^4 t)^2}\right),
\end{align}
and
\begin{align}\label{eq:tan-a-prime-over-a}
        \tan t\frac{a_\veps'}{a_\veps}
        =
        1+
        \frac{2\veps \sin^4t}{1+\veps\sin^4t}.
\end{align}
We shall also use the elementary estimate
\begin{align}\label{eq:ricci-elementary-bound}
        f(q):=
        \frac{q (10-12q )+2\veps q^3  (3-4q )}{(1+\veps q^2)^2}
        \leq \frac{25}{12}
\end{align}
for $q\in [0,1]$ and $\veps\geq0$. 
% Indeed,
% \begin{align}
%         \frac{25}{12}(1+\veps q^2)^2
%         &-
%         \left(q(10-12q)+2\veps q^3(3-4q)\right)\notag\\
%         &=
%         12\left(q-\frac{5}{12}\right)^2
%         +
%         \veps q^2\left(8q^2-6q+\frac{25}{6}\right)
%         +
%         \frac{25}{12}\veps^2q^4
%         \geq0,
% \end{align}
% because $8q^2-6q+\frac{25}{6}>0$ for all $q\in[0,1]$.
Thus
\begin{align}
 R_t &= \operatorname{Ric}_{g_{\veps}}(\partial_t, \partial_t) \\
 &= \Ric_{B}(\partial_t,\partial_t) - \sum_{i=1}^2 \frac{s_i}{b_i}\operatorname{Hess}_{B}b_i(\partial_t,\partial_t)\\
 &= 0 - \frac{a_\veps''}{a_\veps}-\frac{(d-2)c''}{c} \\
 &= 1-\veps f(\sin^2 t) + (d-2) \\
 &\geq (d-1) -\frac{25\veps}{12}, 
\end{align}
where $B=(0,\pi/2)$, $s_1=\operatorname{dim} \bbS^1=1$, $s_2=\dim \bbS^{d-2}=(d-2)$, $b_1=a_\veps$, and $b_2=c$. Similarly, applying Lemma \ref{lem:multiply-warped-ricci}  to the
\(\bbS^1\)-factor with \(B=(0,\pi/2)\), \(F_1=\bbS^1\), \(F_2=\bbS^{d-2}\),
\(b_1=a_\veps\), \(b_2=c\), \(s_1=\dim\bbS^1=1\), and
\(s_2=\dim\bbS^{d-2}=d-2\) and using \eqref{eq:tan-a-prime-over-a}, we have
\begin{align}
    R_\theta
    &=
    \Ric_{g_\veps}(e_\theta,e_\theta) =
    \frac{1}{a_\veps^2}
    \Ric_{g_\veps}(\partial_\theta,\partial_\theta) \notag\\
    &=
    \frac{1}{a_\veps^2}
    \left[
    \Ric_{\bbS^1}(\partial_\theta,\partial_\theta)
    -
    g_\veps(\partial_\theta,\partial_\theta)
    \left(
    \frac{a_\veps''}{a_\veps}
    +
    (d-2)\frac{a_\veps'c'}{a_\veps c}
    \right)
    \right] \notag\\
    &=
    -\frac{a_\veps''}{a_\veps}
    -(d-2)\frac{a_\veps'c'}{a_\veps c} \notag\\
    &=
    -\frac{a_\veps''}{a_\veps}
    +(d-2)\tan t \frac{a_\veps'}{a_\veps} \notag\\
    &=
    (d-2)
    +
    \frac{2(d-2)\veps \sin^4 t}{1+\veps \sin^4 t}
    +
    1
    -
    \veps f(\sin^2 t) \notag\\
    &\geq d-1-\frac{25\veps}{12}.
\end{align}
Lastly, applying Lemma
\ref{lem:multiply-warped-ricci} to the \(\bbS^{d-2}\)-factor gives
\begin{align}
    R_S
    &=
    \Ric_{g_\veps}(e_j,e_j) =
    \frac{1}{c^2}\Ric_{g_\veps}(v_j,v_j) \notag\\
    &=
    \frac{1}{c^2}
    \left[
    \Ric_{\bbS^{d-2}}(v_j,v_j)
    -
    g_\veps(v_j,v_j)
    \left(
    \frac{c''}{c}
    +
    (d-3)\frac{(c')^2}{c^2}
    +
    \frac{a_\veps'c'}{a_\veps c}
    \right)
    \right] \notag\\
    &=
    \frac{d-3}{c^2}
    -\frac{c''}{c}
    -(d-3)\frac{(c')^2}{c^2}
    -\frac{a_\veps'c'}{a_\veps c} \notag\\
    &=
    d-2+\tan t\frac{a_\veps'}{a_\veps} \notag\\
    &=
    d-2+1+\frac{2\veps\sin^4t}{1+\veps\sin^4t}
    \geq d-1.
\end{align}
% Similarly, using \eqref{eq:tan-a-prime-over-a}, we have 
% \begin{align}
%     R_\theta &=  -\frac{a_\veps''}{a_\veps} + (d-2)\tan t \frac{a_\veps'}{a_\veps}\\
%     &= (d-2) + \frac{2(d-2)\veps \sin^4 t}{1+\veps \sin^4 t} + 1  \\
%     &\quad  -\veps f(\sin^2 t)\\
%     &\geq d-1 -\frac{25\veps}{12}.
% \end{align}
% For $v_j$ with $|v_j|_{g_{\bbS^{d-2}}}=1$, the corresponding
% $g_\veps$-unit vector is $e_j=c^{-1}v_j$. Applying Lemma
% \ref{lem:multiply-warped-ricci} to the $\bbS^{d-2}$-factor gives
% \begin{align}
%     R_S
%     &=\operatorname{Ric}_{g_\veps}(e_j,e_j)\notag\\
%     &=
%     \frac{d-3}{c^2}
%     -\frac{c''}{c}
%     -(d-3)\frac{(c')^2}{c^2}
%     -\frac{a_\veps'c'}{a_\veps c}\notag\\
%     &=
%     d-2+\tan t\frac{a_\veps'}{a_\veps}\notag\\
%     &=
%     d-2+1+\frac{2\veps\sin^4t}{1+\veps\sin^4t}
%     \geq d-1.
% \end{align}
Combining the above lower bounds on $R_t$, $R_\theta$ and $R_S$ implies that
\begin{align}
        \operatorname{Ric}_{g_\varepsilon}
        \geq
        \left(
        d-1-\frac{25 \varepsilon}{12}
        \right)g_\varepsilon
        \qquad\text{on }U.
\end{align}
By continuity the above inequality holds on all of $\bbS^d$, i.e.
\begin{equation}\label{eq:ricci-lower}
        \operatorname{Ric}_{g_\varepsilon}
        \geq
        \left(
        d-1-\frac{25 \varepsilon}{12}
        \right)g_\varepsilon.
\end{equation}
Set $\rho_\veps:=d-1-\frac{25 \varepsilon}{12}.$ For $d\geq 2$, $\rho_\veps>0$ for $0<\veps\ll 1$.  

\paragraph{$\mathrm{(iii)}$ \textit{Construction of test functions}} We next show the desired spectral inequality. To achieve this goal, we will construct suitable test functions.  Consider the Rayleigh quotient,
\EQ{
\calR_\varepsilon(f)
    = \frac{\int_{\mathbb S^d}
        |\nabla f|_{g_\varepsilon}^2\,\dvol_{g_\varepsilon}}
        {\int_{\mathbb S^d}
        |f|^2\,\dvol_{g_\varepsilon}}.
}
The simplest test functions to try would be to take functions of the form $|z|^k$, however to take advantage of the $\bbS^1$ factor in the metric, we will consider for $k\geq 1$, the function
\begin{align}
     f_k=\operatorname{Re}(z^k) = q^{k/2}\cos(k\theta),
        \text{ where } q=\sin^2t.   
\end{align}
The $\cos(k\theta)$ term in $f_k$ will turn out to be advantageous. Observe that
\begin{align}\label{eqn:denom-lb}
\int_{\mathbb S^d} |f_k|^2\,\dvol_{g_\varepsilon} &= \int_{\bbS^d} q^{k} \cos^2(k\theta) \sqrt{1+\veps q^2} \dvol_{g_0}\\
&= \frac{1}{2}\int_{\bbS^d} q^{k}  \sqrt{1+\veps q^2} \dvol_{g_0}
\end{align}
where we used the fact that $\dvol_{g_\varepsilon}=\sqrt{1+\varepsilon q^2}\,\dvol_{g_0},$ and the fact that $\int_0^{2\pi} \cos^2(k\theta) =\pi=\frac{1}{2}\int_{0}^{2\pi}\ud\theta$. Then denote $Q_j=\int_{\mathbb{S}^d} q^j \dvol_{g_0}.$ Using $\sqrt{1+x}\geq 1+\frac{x}{2}-\frac{x^2}{8}$ and the above expression, we get 
\begin{align}
  \int_{\mathbb S^d} |f_k|^2\,\dvol_{g_\varepsilon} &\geq \frac{Q_k}{2} +\frac{\veps Q_{k+2}}{4}-\frac{\veps^2 Q_{k+4}}{16}.
\end{align}
Next, we estimate the numerator. Since, $q=\sin^2 t$ we have
\[
        \partial_t f_k
        =
        kq^{(k-1)/2}\sqrt{1-q}\,\cos(k\theta),
        \qquad
        \partial_\theta f_k
        =
        -kq^{k/2}\sin(k\theta),
\]
we have 
\[
|\nabla f_k|_{g_\varepsilon}^2
=
k^2q^{k-1}(1-q)\cos^2(k\theta)
+
\frac{k^2q^{k-1}}{1+\varepsilon q^2}\sin^2(k\theta).
\]
Thus,
\begin{align}\label{eqn:numerator-ub}
\int_{\mathbb S^d} |\nabla f_k|_{g_\veps}^2\,\dvol_{g_\varepsilon} &= \int_{\bbS^d} k^2 q^{k-1}(1-q)\cos^2(k\theta)\sqrt{1+\veps q^2} \dvol_{g_0}  \\
&\quad + \int_{\bbS^d}\frac{k^2 q^{k-1}}{1+\veps q^2} \sin^2 (k\theta) \sqrt{1+\veps q^2} \dvol_{g_0}\\
&= \frac{k^2}{2}\int_{\bbS^d} q^{k-1}(1-q)\sqrt{1+\veps q^2} \dvol_{g_0}  + \frac{k^2}{2}\int_{\bbS^d}\frac{q^{k-1}}{\sqrt{1+\veps q^2}} \dvol_{g_0}\\
&\leq \frac{k^2}{2} \left(2Q_{k-1}-Q_k -\frac{\veps}{2}Q_{k+2}+\frac{3\veps^2}{8}Q_{k+3}\right)\\
&\leq \frac{1}{2}k(k+d-1)Q_k - \frac{\veps k^2}{4} Q_{k+2} + \frac{3\veps^2 k^2}{16} Q_{k+3}
\end{align}
where we used the elementary inequality for $q\in [0,1]$,
\begin{align}
    (1-q)\sqrt{1+\veps q^2} + \frac{1}{\sqrt{1+\veps q^2}}\leq 2-q-\frac{\veps q^3}{2} + \frac{3\veps^2q^4}{8}
\end{align}
and the fact that the sum of the squares of the first two coordinate functions on \(\bbS^d\),
namely \(q=|z|^2\), follows a beta distribution $\operatorname{Beta}(1,\frac{d-1}{2})$ (cf. Corollary 1.1 \cite{frankl1990some}). This implies, in particular, that (cf. Equation~(2.2) in \cite{szablowski2021moment})
\begin{align}
    \frac{Q_k}{Q_{k-1}} =\frac{2k}{2k+d-1},\quad 2Q_{k-1}-Q_k  =\left(\frac{k+d-1}{k}\right) Q_k.
\end{align}
Define
\begin{align}
        \beta_{d,k}
        :=
        \frac{k^2+k(k+d-1)}{2k(k+d-1)}
        \frac{Q_{k+2}}{Q_k}.
\end{align}
Since
\begin{align}
        \frac{Q_{k+2}}{Q_k}
        =
        \frac{2(k+1)}{2(k+1)+d-1}\cdot
        \frac{2(k+2)}{2(k+2)+d-1}
        \to 1, \text{ as }k\to \infty
\end{align}
we see that $\lim_{k\to\infty}\beta_{d,k}=1$. Therefore, since $d\geq4$, we may choose $k\geq k_0$ for some $k_0\gg 1$ such that
\begin{align}
        \eta_{d,k}:=\beta_{d,k}-\frac{25}{12(d-1)}>0.
\end{align}
We fix this value of $k$ for the remainder of the proof. In particular, the integer $K$ defined below depends only on $d$ and not on $\veps$. For some fixed \(k\geq k_0\), define
\begin{align}
        \Phi_{d,k}(\veps)
        :=
        \frac{
        \frac{1}{2}k(k+d-1)Q_k-\frac{\veps k^2}{4}Q_{k+2}
        +\frac{3\veps^2k^2}{16}Q_{k+3}}
        {\frac{Q_k}{2}+\frac{\veps Q_{k+2}}{4}-\frac{\veps^2Q_{k+4}}{16}}.
\end{align}
For $\veps>0$ small enough, the denominator is positive and
\eqref{eqn:denom-lb} and \eqref{eqn:numerator-ub} imply
\begin{align}
        \cR_\veps(f_k)\leq \Phi_{d,k}(\veps).
\end{align}
Moreover, $\Phi_{d,k}$ is smooth near $\veps=0$ and
\begin{align}
        \Phi_{d,k}(0)&=k(k+d-1),\\
        \Phi_{d,k}'(0)&=
        -\frac{k^2+k(k+d-1)}{2}\frac{Q_{k+2}}{Q_k}
        =-k(k+d-1)\beta_{d,k}.
\end{align}
Thus there is a constant $C_{d,k}<\infty$ such that, for all sufficiently small
$\veps>0$,
\begin{align}
        \cR_\veps(f_k)
        \leq
        k(k+d-1)(1-\beta_{d,k}\veps)+C_{d,k}\veps^2.
\end{align}
Choosing $\veps>0$ further so that
\begin{align}
        C_{d,k}\veps^2
        <
        \frac12 k(k+d-1)\eta_{d,k}\veps,
\end{align}
we get
\begin{align}\label{eqn:f_k-ub}
        \mathcal{R}_\veps(f_k)
        &<
        k(k+d-1)\left(1-\frac{25\veps}{12(d-1)}\right)\\
        &=
        \frac{\rho_\veps}{d-1}k(k+d-1),
\end{align}
where we recall that $\rho_\veps:=d-1-\frac{25\veps}{12}.$ 
Define the following real vector spaces
\begin{align}
  E_{k-1}:=\bigoplus_{\ell=0}^{k-1}H_{\ell},\quad F_k:= E_{k-1} \oplus \operatorname{span}\{f_k\}  
\end{align}
where $H_\ell$ is the space of spherical harmonics of degree $0\leq \ell\leq k-1$. Note that $f_k$ is orthogonal to $H_\ell$ for any $0\leq \ell\leq k-1$ since, after complexifying $H_\ell$, any $h_{\ell}\in H_\ell$ can be expressed as a linear combination of $z^\alp \bar{z}^\beta y^\gamma$ for $\alp,\beta\in \bN_0$ and $\gamma \in \bN_0^{d-1}$ and some constants $c_{\alp\beta\gamma}\in \bC$ as follows:
\begin{align}
    \int_{\mathbb{S}^d} h_{\ell} f_k \dvol_{g_\veps} &= \sum_{\alpha+\beta+|\gamma|=\ell}c_{\alpha\beta\gamma}\int_{\mathbb{S}^d} z^\alp \bar{z}^\beta y^\gamma  \Re(z^k) \dvol_{g_\veps} \\
    &= \sum_{\alpha+\beta+|\gamma|=\ell}\frac{c_{\alpha\beta\gamma}}{2}\int_{\mathbb{S}^d} z^\alp \bar{z}^\beta y^\gamma (z^k+\bar{z}^k)   \dvol_{g_\veps}\\
    &= \frac12
    \sum_{\alpha+\beta+|\gamma|=\ell}
    c_{\alpha\beta\gamma}
    \left[
    \int_{\mathbb S^d}
    z^{\alpha+k}\bar z^\beta y^\gamma
    \,d\operatorname{vol}_{g_\veps}
    +
    \int_{\mathbb S^d}
    z^\alpha\bar z^{\beta+k}y^\gamma
    \,d\operatorname{vol}_{g_\veps}
    \right].
\end{align}
Both of the above integrals vanish. This can be seen by making the change of variables
\((z,y)\mapsto(e^{i\tau}z,y)\) which preserves the volume form $\dvol_{g_{\veps}}$
\begin{align}\label{eqn:change-of-variables}
    \int_{\mathbb S^d}
    z^{\alpha+k}\bar z^\beta y^\gamma
    \,\dvol_{g_\veps}
    &=
    e^{i(\alpha+k-\beta)\tau}
    \int_{\mathbb S^d}
    z^{\alpha+k}\bar z^\beta y^\gamma
    \,\dvol_{g_\veps}.
\end{align}
Since $\alpha+\beta+|\gamma|=\ell\leq k-1,$  we get $\alpha+k-\beta\neq 0.$ Choosing
\(\tau\) so that \(e^{i(\alpha+k-\beta)\tau}\neq 1\), we obtain
\[
        \int_{\mathbb S^d}
        z^{\alpha+k}\bar z^\beta y^\gamma
        \,\dvol_{g_\veps}
        =
        0.
\]
Similarly,
\begin{align}
    \int_{\mathbb S^d}
    z^\alpha\bar z^{\beta+k}y^\gamma
    \,\dvol_{g_\veps}
    &= 0.
\end{align}
Thus, for any $0\leq \ell \leq k-1$, we have
\begin{align}
\int_{\mathbb{S}^d} h_{\ell} f_k \dvol_{g_\veps} = 0.
\end{align}
We now claim that 
\begin{align}\label{claim:F_k}
    \sup_{f\neq 0,f\in F_k}\mathcal{R}_\veps(f)<\frac{\rho_\veps}{d-1}k(k+d-1).
\end{align}
To this end, take any \(f\in F_k\) and write
\(f=h_{k-1}+a f_k\), where \(a\in\R\) and \(h_{k-1}\in E_{k-1}\). Then
\begin{align}
  \cR_\veps(f) &= \mathcal{R}_\veps(h_{k-1} + a f_k)\\
  &=  \frac{\int_{\mathbb S^d}
        |\nabla h_{k-1}|_{g_\varepsilon}^2\,\dvol_{g_\varepsilon}+a^2\int_{\mathbb S^d}
        |\nabla f_{k}|_{g_\varepsilon}^2\,\dvol_{g_\varepsilon}}
        {\int_{\mathbb S^d}
        |h_{k-1}|^2\,\dvol_{g_\varepsilon}+a^2\int_{\mathbb S^d}
        |f_{k}|^2\,\dvol_{g_\varepsilon}}\\
        &\leq \max \{\cR_\veps(h_{k-1}),\cR_\veps(f_k)\},
\end{align}
where we used the fact that for any $h_{k-1}\in E_{k-1}$,
\begin{align}
    &\int_{\mathbb{S}^d} \la \nabla h_{k-1}, \nabla f_k \ra_{g_\veps} \dvol_{g_\veps} \\
    &\quad = \int_{\mathbb{S}^d} \left( (\partial_t h_{k-1})(\partial_t f_k)  + \frac{1}{q(1+\veps q^2)}(\partial_\theta h_{k-1})(\partial_\theta f_k) + \frac{1}{\cos^2 t}\la \nabla_\om h_{k-1}, \nabla_\om f_k\ra \right) \dvol_{g_\veps}\\
    &\quad = 0.
\end{align}
Indeed, $f_k$ is independent of $\omega$, so the last term vanishes. For a monomial
$z^\alpha\bar z^\beta y^\gamma$ of degree $\ell\leq k-1$, the $\partial_t$- and
$\partial_\theta$-terms above are finite sums of terms with $\theta$-weights
$e^{i(\alpha-\beta+k)\theta}$ or $e^{i(\alpha-\beta-k)\theta}$, multiplied by
functions of $t$ and $\omega$ only. Since $\alpha+\beta+|\gamma|=\ell\leq k-1$, both
$\alpha-\beta+k$ and $\alpha-\beta-k$ are nonzero. Their $\theta$-integrals vanish,
and the claim follows by linearity. We already have an upper bound for $\cR_\veps(f_k)$ so we focus on $\cR_\veps(h_{k-1}).$ To this end, first observe that for the Rayleigh quotient for the canonical round metric on the sphere $\bbS^d$, we have 
\begin{align}
    \cR_0(h)\leq (k-1)(k+d-2)
\end{align}
for any $h\in E_{k-1}$ since 
\begin{align}
    -\Delta_{g_0} h_{\ell} = \ell(\ell + d-1) h_{\ell}
\end{align}
for any $h_{\ell} \in H_\ell$ and $0\leq \ell\leq k-1.$ Furthermore, for any function $h\in E_{k-1}$ and for sufficiently small $0<\veps\ll 1$ we have
\begin{align}\label{eqn:h-ub}
    \cR_\veps(h) &= \frac{\int_{\bbS^d} |\nabla h|_{g_\veps}^2\dvol_{g_\veps}}{\int_{\bbS^d} h^2 \dvol_{g_\veps}}  \\
    &\leq \sqrt{1+\veps} \frac{\int_{\bbS^d} |\nabla h|_{g_0}^2\dvol_{g_0}}{\int_{\bbS^d} h^2 \dvol_{g_0}}  \\
    &= \sqrt{1+\veps}\cR_0(h)\\
    &\leq \left(1+\frac{\veps}{2}\right) (k-1)(k+d-2)\\
    &< \frac{\rho_\veps}{d-1}k(k+d-1),
\end{align}
where the last inequality follows by choosing $\veps<\veps_0$ where $\veps_0:=\frac{2k+d-2}{\frac{1}{2}(k-1)(k+d-2)+\frac{25}{12(d-1)}k(k+d-1)}$. Therefore, by \eqref{eqn:f_k-ub} and \eqref{eqn:h-ub} we have 
\begin{align}
    \sup_{f\neq 0,f\in F_k}\mathcal{R}_\veps(f)<\frac{\rho_\veps}{d-1}k(k+d-1).
\end{align}
\paragraph{$\mathrm{(iv)}$ \textit{Rescaling and convergence to the round metric}} Define
\begin{align}
        K
        :=
        \operatorname{dim}E_{k-1}+1,
        \quad
        c_\veps
        :=
        \frac{\rho_\veps}{d-1}
        =
        1-\frac{25\veps}{12(d-1)},\quad 
        \widehat g_\veps
        :=
        c_\veps g_\veps.
\end{align}
Since $E_{k-1}$ is the direct sum of the eigenspaces of the round sphere of degrees $0,\ldots,k-1$, the integer $K$ is the first index in the degree-$k$ eigenvalue cluster. Hence
\begin{align}\label{eqn:sphere-round-K}
        \lam_K(\bbS^d,g_0)
        =
        k(k+d-1).
\end{align}
It follows from the min--max characterization and the estimate above on
$F_k$ that
\begin{align}\label{eqn:sphere-unscaled-local-gap}
        \lam_K(\bbS^d,g_\veps)
        \leq
        \sup_{0\neq f\in F_k}\cR_\veps(f)
        <
        \frac{\rho_\veps}{d-1}k(k+d-1)
        =
        c_\veps\lam_K(\bbS^d,g_0).
\end{align}

Choose $\veps_d>0$ sufficiently small so that all the preceding estimates
hold and $c_\veps>0$ for every $0\leq\veps<\veps_d$. Then the metrics
$\widehat g_\veps$ are well-defined. Since $c_\veps$ is constant on
$\bbS^d$, the Levi--Civita connections of $g_\veps$ and
$\widehat g_\veps$ agree. Therefore, using
\eqref{eq:ricci-lower}, we obtain
\begin{align}
        \Ric_{\widehat g_\veps}
        =
        \Ric_{g_\veps}\geq
        \rho_\veps g_\veps=
        (d-1)\widehat g_\veps.
\end{align}
This proves \eqref{eqn:sphere-local-Ricci}. Moreover, for every constant $c>0$, we have $ \cR_{cg}(f)=c^{-1}\cR_g(f).$ 
Consequently, the min--max characterization gives
$\lam_j(\bbS^d,cg)=c^{-1}\lam_j(\bbS^d,g)$ for every $j\geq1$. Thus,
using \eqref{eqn:sphere-unscaled-local-gap}, we obtain
\begin{align}
        \lam_K(\bbS^d,\widehat g_\veps)
        =
        c_\veps^{-1}\lam_K(\bbS^d,g_\veps)<
        \lam_K(\bbS^d,g_0).
\end{align}
This proves \eqref{eqn:index-K}. Finally, by \eqref{eq:geps-global}, we have
\begin{align}
        \widehat g_\veps-g_0
        &=
        \left(
        1-\frac{25\veps}{12(d-1)}
        \right)
        \left(
        g_0+\veps |z|^2\alpha\otimes\alpha
        \right)
        -g_0\\
        &=
        \veps
        \left(
        |z|^2\alpha\otimes\alpha
        -\frac{25}{12(d-1)}g_0
        \right)
        -
        \frac{25\veps^2}{12(d-1)}
        |z|^2\alpha\otimes\alpha.
\end{align}
Since $|z|^2\alpha\otimes\alpha$ is a smooth tensor on the compact sphere,
for every integer $m\geq0$ there exists a constant $C_{d,m}<\infty$ such
that $\left\|
        \widehat g_\veps-g_0
        \right\|_{C^m(g_0)}
        \leq
        C_{d,m}\veps.$
Consequently,
\begin{align}
        \widehat g_\veps
        \rightarrow
        g_0
        =
        g_\can
         \text{ in }C^\infty
        \text{ as }\veps\to0.
\end{align}
This proves \eqref{eqn:sphere-local-convergence} and completes the proof.
\end{proof}
% \begin{cor}\label{cor:sphere}
% Let $(\bbS^d,g)$ be the metric constructed in Theorem~\ref{thm:sphere-main}. Then there does not exist a map $T:(\bbS^d,g_\can,\dvol_{g_\can})\to (\bbS^d,g,\dvol_g)$ which is measure preserving and contracting.  
% \end{cor}
\begin{cor}\label{cor:sphere}
Let $\{\widehat g_\veps\}_{0\leq\veps<\veps_d}$ be the family of metrics
constructed in Theorem~\ref{thm:sphere-main}. Then, for every
$0<\veps<\veps_d$, there does not exist a $1$-Lipschitz map
\[
        T:
        (\bbS^d,g_\can,\vol_{g_\can})
        \to
        (\bbS^d,\widehat g_\veps,\vol_{\widehat g_\veps})
\]
pushing forward $\vol_{g_\can}$ onto $\vol_{\widehat g_\veps}$ up to a
finite constant.
\end{cor}

\begin{proof}
If such a map existed, then Theorem~\ref{thm:contraction} would imply that for every $k\geq1$ we have
\begin{align}
        \lam_k(\bbS^d,\widehat g_\veps)
        \geq
        \lam_k(\bbS^d,g_\can).
\end{align}
This contradicts \eqref{eqn:index-K}.
\end{proof}
Thus, Theorem~\ref{thm:sphere-main} and Corollary \ref{cor:sphere} imply that Conjectures 3 and 4 in \cite{milman2018spectral} are not true in general.
\section{Proof of Theorem~\ref{thm:gaussian-main}}\label{sec:proof-of-gaussian-main}
The idea behind Theorem~\ref{thm:gaussian-main} is to deform the Euclidean
metric in polar coordinates.  Recall that the standard metric on \(\mathbb R^d\)
is
\[
        \ud r^2+r^2g_{\bbS^{d-1}},
\]
so the spherical slices \(\{r=\mathrm{constant}\}\) have radius \(r\) and keep
expanding as \(r\) increases.  In our construction, we instead choose a
rotationally symmetric metric for which these slices expand in the usual way
near the origin, so that the metric closes smoothly, but then their radius
becomes constant and equal to \(\sqrt{d-2}\).  Thus the end of the manifold is
cylindrical, with cross-section \(\bbS^{d-1}(\sqrt{d-2})\).  This cylinder is the
source of the spectral obstruction: its cross-section has first nonzero
eigenvalue
\begin{align}
        \frac{d-1}{d-2}<2,
\end{align}
attained by the \(d\) coordinate functions restricted to
\(\bbS^{d-1}(\sqrt{d-2})\), while the cylindrical direction produces one
additional low-energy radial mode.  Finally, we put a quadratic potential in the
cylindrical direction, so that the end carries a Gaussian weight.  The weight makes the total measure of the space finite and provides enough
curvature for the curvature--dimension condition
\(\operatorname{CD}(1,\infty)\). Thus, the constructed weighted
manifold satisfies the curvature-dimension lower bound, but its cylindrical end
produces \(d+1\) nonzero test functions with Rayleigh quotient below the
Gaussian threshold \(2\), forcing
\begin{align}
        \lambda_{d+2}(\mathbb R^d,g,\mu)
        <
        2
        =
        \lambda_{d+2}(\mathbb R^d,|\cdot|,\gamma^d).
\end{align}
Before proceeding with the proof of Theorem~\ref{thm:gaussian-main}, we record a
technical smoothing lemma used to construct the torpedo function (cf. \cite{botvinnik2010homotopy}).
\begin{lem}\label{lem:theta-eps}
Let $\delta>0$, $B=\frac{\pi\delta}{2}$, and $0<\veps<B/2$. There exists a smooth function
$\te_\veps:[0,\infty)\to[0,\pi/2]$ such that
\begin{align}
        \te_\veps(r)
        &=
        \frac{r}{\delta}
        \qquad\text{when }0\leq r\leq B-\veps,\\
        \te_\veps(r)
        &=
        \frac{\pi}{2}
        \qquad\text{when }r\geq B+\veps,
\end{align}
with $\te_\veps'\geq0$, $\te_\veps'\leq\delta^{-1}$, and $\te_\veps''\leq0$.
\end{lem}
\begin{proof}
Define
\[
        \psi(s)
        =
        \begin{cases}
        \exp\!\left(-\dfrac{1}{1-4s^2}\right),& |s|<\dfrac12,\\[6pt]
        0,& |s|\geq \dfrac12.
        \end{cases}
\]
Then \(\psi\in C_c^\infty((-1,1))\), \(\psi\geq0\), \(\psi\) is even, and
\(\psi\) is not identically zero. Set
\[
        D:=\int_{-1}^1\psi(u)\,\ud u>0
\]
and define
\begin{align}
        \eta(s)
        =
        \frac{1}{D}\int_s^1\psi(u)\,\ud u .
\end{align}
Since \(\psi\geq0\), we have \(0\leq\eta\leq1\), and
\[
        \eta'(s)=-\frac{\psi(s)}{D}\leq0.
\]
Moreover, since \(\psi\) is supported in \([-1/2,1/2]\), we have
\[
        \eta(s)=1\quad\text{for }s\leq -\frac12,
        \qquad
        \eta(s)=0\quad\text{for }s\geq \frac12.
\]
Thus \(\eta\) is constant near \(s=-1\) and near \(s=1\). Finally, since
\(\psi\) is even, we have
\[
        \eta(-s)=1-\eta(s).
\]
Therefore,
\[
        \int_{-1}^1\eta(s)\,\ud s
        =
        \frac12\int_{-1}^1\bigl(\eta(s)+\eta(-s)\bigr)\,\ud s
        =
        1.
\]

Define
\begin{align}
        \te_\veps(r)=
        \begin{cases}
        \dfrac{r}{\delta},&0\leq r\leq B-\veps,\\[6pt]
        \dfrac{B-\veps}{\delta}
        +\dfrac{\veps}{\delta}
        \displaystyle\int_{-1}^{(r-B)/\veps}\eta(s)\,\ud s,
        &B-\veps\leq r\leq B+\veps,\\[12pt]
        \dfrac{\pi}{2},&r\geq B+\veps.
        \end{cases}
\end{align}
When \(B-\veps\leq r\leq B+\veps\), we have
\[
        \te_\veps'(r)
        =
        \frac{1}{\delta}\eta\!\left(\frac{r-B}{\veps}\right),
        \qquad
        \te_\veps''(r)
        =
        \frac{1}{\delta\veps}
        \eta'\!\left(\frac{r-B}{\veps}\right)
        \leq0.
\]
Since \(0\leq\eta\leq1\), this gives
\[
        0\leq \te_\veps'(r)\leq \delta^{-1}.
\]
The same inequalities are immediate on the two outer regions, where
\(\te_\veps(r)=r/\delta\) and \(\te_\veps(r)=\pi/2\), respectively. Thus
\(\te_\veps'\geq0\), \(\te_\veps'\leq\delta^{-1}\), and
\(\te_\veps''\leq0\) on all of \([0,\infty)\).
\end{proof}

\begin{proof}[Proof of Theorem~\ref{thm:gaussian-main}]
Let \(d\geq4\) and let \(0<\veps\ll1\). The argument proceeds in a similar style as in the proof of Theorem~\ref{thm:sphere-main}. 
\paragraph{$\mathrm{(i)}$ \textit{Construction of the metric}} Set
\begin{align}
        \delta=\sqrt{d-2},
        \qquad
        B=\frac{\pi\delta}{2}.
\end{align}
Following the notation of \cite[Definition 2.4]{botvinnik2010homotopy}, we use a
torpedo function \(\rho_\delta\). Since we require more precise estimates on
\(\rho_{\delta,\veps}\), we define it explicitly. Let
\(\te_\veps:[0,\infty)\to[0,\pi/2]\) be the function from Lemma
\ref{lem:theta-eps}, and define
\begin{align}
        \rho_{\delta,\veps}(r)=\delta\sin\te_\veps(r).
\end{align}
Then
\begin{align}
        \rho_{\delta,\veps}(r)
        &=
        \delta\sin(r/\delta)
        \qquad\text{when }0\leq r\leq B-\veps,\\
        \rho_{\delta,\veps}(r)
        &=
        \delta
        \qquad\qquad\qquad\text{when }r\geq B+\veps.
\end{align}
Moreover, since
\begin{align}
        \rho_{\delta,\veps}''
        =
        \delta\cos\te_\veps\,\te_\veps''
        -
        \delta\sin\te_\veps\,(\te_\veps')^2,
\end{align}
and $0\leq\te_\veps\leq\pi/2$, we have \(\rho_{\delta,\veps}''\leq0\). Thus
\(\rho_{\delta,\veps}\) is a torpedo function in the sense of
\cite[Definition 2.4]{botvinnik2010homotopy}, except that we keep track of the
transition scale \(\veps\).
We now define the metric directly on \(\bbR^d\). On \(\bbR^d\setminus\{0\}\),
write
\begin{align}
        x=r\omega,
        \qquad
        r=|x|,
        \qquad
        \omega\in\bbS^{d-1}.
\end{align}
Define
\begin{align}\label{eq:noncompact-torpedo-metric}
        g_\veps
        =
        \ud r^2+\rho_{\delta,\veps}(r)^2g_{\bbS^{d-1}},
\end{align}
where \(g_{\bbS^{d-1}}\) is the canonical round metric on the unit sphere
\(\bbS^{d-1}\). As \(r\to0\) we have
\begin{align}
        \rho_{\delta,\veps}(r)
        =
        \delta\sin(r/\delta)
        =
        r-\frac{r^3}{6\delta^2}+O(r^5),
\end{align}
the metric smoothly extends across the origin as a rotationally symmetric metric
on \(\bbR^d\). Moreover,
\begin{align}
        g_\veps
        =
        \ud r^2+\delta^2g_{\bbS^{d-1}}
        \qquad\text{when }r\geq B+\veps.
\end{align}
This also shows that \(g_\veps\) is cylindrical at infinity and hence complete. Moreover, Lemma~\ref{lem:theta-eps} gives
\begin{align}
    0\leq \rho_{\delta,\veps}'
    =\delta\cos\te_\veps\,\te_\veps'
    \leq1.
\end{align}
Together with $\rho_{\delta,\veps}''\leq0$, we can compute the sectional curvature for warped product metrics as in \cite[Section~4.2.3]{petersen2006riemannian} to deduce that for every $p=r\omega\in\bbR^d\setminus\{0\}$ and
$g_\veps$-orthonormal vectors
$X,Y\in T_p\bbS_r^{d-1}\subset T_p\bbR^d$, where
$\bbS_r^{d-1}:=\{x\in\bbR^d:|x|=r\}$, we have
\begin{align}
    K_{g_\veps}(\partial_r,X)
    =-\frac{\rho_{\delta,\veps}''}
    {\rho_{\delta,\veps}}\geq0,\quad 
    K_{g_\veps}(X,Y)
    =\frac{1-(\rho_{\delta,\veps}')^2}
    {\rho_{\delta,\veps}^2}\geq0.
\end{align}
Here $K_{g_\veps}(U,V)$ denotes the sectional curvature
of $\operatorname{span}\{U,V\}\subset T_p\bbR^d$
with respect to $g_\veps$. For this rotationally symmetric metric, the curvature operator is diagonal on radial and tangential two-forms, with the displayed quantities as its eigenvalues. Thus $\operatorname{Sec}_{g_\veps}\geq0$ for $r>0$, and the same bound holds at the origin by smoothness.

\paragraph{$\mathrm{(ii)}$ \textit{Construction of the measure}} Next, we will define a suitable weight which will be quadratic on the cylindrical end. To this end, choose a smooth function \(q_\veps:[0,\infty)\to[0,1]\) such
that
\begin{align}
        q_\veps=0\quad\text{on }[0,B-2\veps],
        \qquad
        q_\veps=1\quad\text{on }[B-\veps,\infty).
\end{align}
Then define \(V_\veps\) as
\begin{align}
        V_\veps(r)=\int_0^r(r-u)q_\veps(u)\,\ud u,
        \qquad
        V_\veps\geq0,
        \qquad
        V_\veps'\geq0.
\end{align}
Set
\begin{align}
        \nu_\veps=e^{-V_\veps}\dvol_{g_\veps},
        \qquad
        Z_\veps=\int_{\bbR^d}\ud\nu_\veps,
        \qquad
        \mu_\veps=Z_\veps^{-1}\nu_\veps.
\end{align}
Since \(V_\veps\) is bounded from below by a quadratic function in \(r\) for
\(r\geq B-\veps\), the normalizing constant \(Z_\veps<\infty\).

\paragraph{$\mathrm{(iii)}$ \textit{Verification of the $\CD(1,\infty)$ condition}} We now show that \((\bbR^d,g_\veps,\mu_\veps)\) is \(\CD(1,\infty)\). Let
\(\partial_r\) denote the unit vector in the radial direction and let
\(X\perp\partial_r\) be a \(g_\veps\)-unit vector. Since the potential
\(V_\veps=V_\veps(r)\) is radial, the Bakry--\'Emery tensor
\begin{align}
        \Ric_{g_\veps,V_\veps}:=\Ric_{g_\veps}+\nabla^2V_\veps
\end{align}
is given by (cf. proof of Proposition 2 in \cite{wylie2016some})
\begin{align}
        \Ric_{g_\veps,V_\veps}(\partial_r,\partial_r)
        &=
        -(d-1)\frac{\rho_{\delta,\veps}''}{\rho_{\delta,\veps}}
        +
        V_\veps'',\\
        \Ric_{g_\veps,V_\veps}(\partial_r,X)
        &=
        0,\\
        \Ric_{g_\veps,V_\veps}(X,X)
        &=
        -\frac{\rho_{\delta,\veps}''}{\rho_{\delta,\veps}}
        +
        (d-2)\frac{(1-(\rho_{\delta,\veps}')^2)}{\rho_{\delta,\veps}^2}
        +
        \frac{V_\veps'\rho_{\delta,\veps}'}{\rho_{\delta,\veps}}.
\end{align}
Since
\begin{align}
        \rho_{\delta,\veps}'\geq0,
        \qquad
        \rho_{\delta,\veps}''\leq0,
\end{align}
and since
\begin{align}
        \rho_{\delta,\veps}'
        =
        \delta\cos\te_\veps\,\te_\veps',
\end{align}
we have from Lemma \ref{lem:theta-eps} that
\begin{align}
        (\rho_{\delta,\veps}')^2
        \leq
        \cos^2\te_\veps
        =
        1-\frac{\rho_{\delta,\veps}^2}{\delta^2}.
\end{align}
Since \(\delta^2=d-2\), we get
\begin{align}
        (d-2)\bigl(1-(\rho_{\delta,\veps}')^2\bigr)
        \geq
        \rho_{\delta,\veps}^2.
\end{align}
Since also \(-\rho_{\delta,\veps}\rho_{\delta,\veps}''\geq0\), we get
\begin{align}
        -\frac{\rho_{\delta,\veps}''}{\rho_{\delta,\veps}}
        +
        (d-2)\frac{1-(\rho_{\delta,\veps}')^2}{\rho_{\delta,\veps}^2}
        &=
        \frac{(d-2)(1-(\rho_{\delta,\veps}')^2)
        -\rho_{\delta,\veps}\rho_{\delta,\veps}''}
        {\rho_{\delta,\veps}^2}\geq1.
\end{align}
Moreover,
\begin{align}
        V_\veps'\frac{\rho_{\delta,\veps}'}{\rho_{\delta,\veps}}\geq0,
\end{align}
because \(V_\veps'\geq0\) and \(\rho_{\delta,\veps}'\geq0\). Thus
\begin{align}
        \Ric_{g_\veps,V_\veps}(X,X)\geq1.
\end{align}
It remains to check the radial direction. If \(0<r\leq B-\veps\), then
\begin{align}
        \rho_{\delta,\veps}(r)
        =
        \delta\sin(r/\delta),
        \qquad
        \rho_{\delta,\veps}''
        =
        -\frac{1}{\delta^2}\rho_{\delta,\veps}.
\end{align}
Hence
\begin{align}
        \Ric_{g_\veps,V_\veps}(\partial_r,\partial_r)
        &=
        -(d-1)\frac{\rho_{\delta,\veps}''}{\rho_{\delta,\veps}}
        +
        V_\veps''\geq
        \frac{d-1}{\delta^2}=
        \frac{d-1}{d-2}
        >
        1.
\end{align}
If \(r\geq B-\veps\), then since
\begin{align}
        V_\veps'(r)
        =
        \int_0^r q_\veps(u)\,\ud u
        \implies
        V_\veps''(r)=q_\veps(r),
\end{align}
and the identity $q_\veps=1$ on $[B-\veps,\infty)$ gives \(V_\veps''(r)=1\). Combining this fact with
\(\rho_{\delta,\veps}''\leq0\), we get
\begin{align}
        \Ric_{g_\veps,V_\veps}(\partial_r,\partial_r)
        =
        -(d-1)\frac{\rho_{\delta,\veps}''}{\rho_{\delta,\veps}}
        +
        1
        \geq1.
\end{align}
Thus, on \(\bbR^d\setminus\{0\}\),
\begin{align}
        \Ric_{g_\veps}+\nabla^2V_\veps\geq g_\veps,
\end{align}
which by continuity extends to $r=0.$ Writing \(W_\veps=V_\veps+\log Z_\veps\), we have
\begin{align}
        \mu_\veps=e^{-W_\veps}\dvol_{g_\veps},
        \qquad
        \Ric_{g_\veps}+\nabla^2W_\veps
        =
        \Ric_{g_\veps}+\nabla^2V_\veps
        \geq g_\veps.
\end{align}
Hence \((\bbR^d,g_\veps,\mu_\veps)\) satisfies \(\CD(1,\infty)\).
\paragraph{$\mathrm{(iv)}$ \textit{Construction of test functions}}
We finally show the desired spectral inequality. To achieve this goal we will
construct suitable test functions. Since the normalization constants cancel in the
Rayleigh quotient, it suffices to use \(\nu_\veps\) instead of \(\mu_\veps\). Thus,
for any nonzero test function \(u\in W^{1,2}(\bbR^d,g_\veps,\nu_\veps)\), set
\begin{align}
        \cR_\veps(u)
        =
        \frac{\int_{\bbR^d}|\nabla u|_{g_\veps}^2\,\ud\nu_\veps}
        {\int_{\bbR^d}u^2\,\ud\nu_\veps}.
\end{align}
We would like to compare \(\cR_\veps\) with a simpler model Rayleigh quotient. To
this end define
\begin{align}
        \rho_0(r)=
        \begin{cases}
        \delta\sin(r/\delta),&0\leq r\leq B,\\
        \delta,&r\geq B,
        \end{cases}
        \qquad
        V_0(r)=
        \begin{cases}
        0,&0\leq r\leq B,\\
        \frac{(r-B)^2}{2},&r\geq B.
        \end{cases}
\end{align}
Let
\begin{align}
        \nu_0=e^{-V_0}\rho_0(r)^{d-1}\,\ud r\,\dvol_{\bbS^{d-1}}.
\end{align}
By construction, \(\rho_{\delta,\veps}=\rho_0\) on
\([0,B-\veps]\cup [B+\veps,\infty)\), while \(V_\veps=V_0\) on
\([0,B-2\veps]\). Next, observe that since \(q_\veps(u)=0\) on
\([0,B-2\veps]\) and \(0\leq q_\veps\leq1\), we have
\begin{align}
        V_\veps(B)
        &=
        \int_{B-2\veps}^{B}(B-u)q_\veps(u)\,\ud u
        \leq 2\veps^2,\\
        V_\veps'(B)
        &=
        \int_{B-2\veps}^{B}q_\veps(u)\,\ud u
        \leq 2\veps.
\end{align}
Furthermore, since \(V_\veps''(r)=1\) for all \(r\geq B\), we get
\begin{align}
        V_\veps'(r)-V_\veps'(B)
        &=
        \int_B^r V_\veps''(t)\,\ud t
        =
        r-B,\\
        V_\veps(r)-V_\veps(B)
        &=
        \int_B^r V_\veps'(t)\,\ud t
        =
        V_\veps'(B)(r-B)+\frac{1}{2}(r-B)^2.
\end{align}
Therefore, using the fact \(V_0(r)=\frac{1}{2}(r-B)^2\) for \(r\geq B\), we get
for \(\veps\leq1\),
\begin{align}
        0\leq V_\veps(r)-V_0(r)
        \leq
        2\veps(1+r-B).
\end{align}
Thus, using the elementary inequality \(1-e^{-x}\leq x\) for all \(x\geq0\), we
get for \(r\geq B\),
\begin{align}\label{eq:end-density-error}
        \left|e^{-V_\veps(r)}-e^{-V_0(r)}\right|
        \leq e^{-V_0(r)} |1-e^{-(V_\veps(r)-V_0(r))}|
        \leq
        2\veps(1+r-B)e^{-(r-B)^2/2}.
\end{align}
Let \(x_1,\dots,x_d\) be the coordinate functions on \(\bbR^d\) restricted to the
unit sphere \(\bbS^{d-1}\). In particular, they satisfy
\begin{align}
        -\Delta_{\bbS^{d-1}}x_i=(d-1)x_i,
        \qquad
        \int_{\bbS^{d-1}}x_i\,\dvol_{\bbS^{d-1}}=0,
\end{align}
and
\begin{align}
        \int_{\bbS^{d-1}}x_ix_j\,\dvol_{\bbS^{d-1}}=0
        \qquad(i\neq j).
\end{align}
Set
\begin{align}
        a_\veps(r)=\frac{\rho_{\delta,\veps}(r)}{\delta},
        \qquad
        U_{\veps,i}(r,\omega)=a_\veps(r)x_i(\omega).
\end{align}
Then \(|a_\veps|\leq1\), \(|a_\veps'|\leq\frac1\delta\), and \(U_{\veps,i}\)
are smooth as \(r\to0\). For \(r>0\),
\begin{align}
        |\nabla U_{\veps,i}|_{g_\veps}^2
        &=
        (a_\veps')^2x_i^2
        +
        \frac{a_\veps^2}{\rho_{\delta,\veps}^2}
        |\nabla_{\bbS^{d-1}}x_i|^2=
        (a_\veps')^2x_i^2
        +
        \frac{1}{\delta^2}|\nabla_{\bbS^{d-1}}x_i|^2.
\end{align}
Define
\begin{align}
        a_0(r)=\frac{\rho_0(r)}{\delta},
        \qquad
        U_{0,i}(r,\omega)=a_0(r)x_i(\omega),
\end{align}
and
\begin{align}
        \calD_U^0=\int U_{0,i}^2\,\ud\nu_0,
        \qquad
        \calN_U^0=\int |\nabla U_{0,i}|^2\,\ud\nu_0,
\end{align}
where \(a_0'\) is understood in the weak sense in \(\calN_U^0\). Denote \(I_\veps:=[B-2\veps,B+\veps]\). Then 
\begin{align}\label{eq:angular-error-estimate-new}
        &\left|
        \int U_{\veps,i}^2\,\ud\nu_\veps-\calD_U^0
        \right|
        +
        \left|
        \int |\nabla U_{\veps,i}|_{g_\veps}^2\,\ud\nu_\veps-\calN_U^0
        \right| \notag\\
        &\leq
        \int_0^\infty\int_{\bbS^{d-1}}
        \left|
        a_\veps^2x_i^2 e^{-V_\veps}\rho_{\delta,\veps}^{d-1}
        -
        a_0^2x_i^2 e^{-V_0}\rho_0^{d-1}
        \right|
        \,\dvol_{\bbS^{d-1}}\ud r \notag\\
        &\quad+
        \int_0^\infty\int_{\bbS^{d-1}}
        \left|
        \left((a_\veps')^2x_i^2+\frac1{\delta^2}|\nabla_{\bbS^{d-1}}x_i|^2\right)
        e^{-V_\veps}\rho_{\delta,\veps}^{d-1}
        -
        \left((a_0')^2x_i^2+\frac1{\delta^2}|\nabla_{\bbS^{d-1}}x_i|^2\right)
        e^{-V_0}\rho_0^{d-1}
        \right|
        \,\dvol_{\bbS^{d-1}}\ud r \notag\\
        &\leq
        C_d |I_\veps|
        +
        C_d\int_B^\infty
        \left|e^{-V_\veps(r)}-e^{-V_0(r)}\right|\,\ud r \notag\\
        &\leq
        C_d\veps
        +
        C_d\veps
        \int_B^\infty
        (1+r-B)e^{-(r-B)^2/2}\,\ud r \notag\\
        &\leq
        C_d\veps,
\end{align}
with constants $C_d>0$ depending on the dimension $d\geq 4$ and possibly different on each line. In the above display, we used that \(x_i\), \(\nabla_{\bbS^{d-1}}x_i\), and the
quantities
\[
        |a_\veps|\leq 1,\qquad |a_\veps'|\leq \frac1\delta,\qquad
        0\leq \rho_{\delta,\veps}\leq \delta,\qquad e^{-V_\veps}\leq 1
\]
are uniformly bounded on \(I_\veps\). Define \(\delta_U:=2\calD_U^0-\calN_U^0\), then 
\begin{align}\label{eqn:notation}
        \omega_{d-1}:=|\mathbb S^{d-1}|,
        \qquad
        I_m:=\int_0^{\pi/2}\sin^m t\,\ud t,
        \qquad
        J_0:=\int_0^\infty e^{-\tau^2/2}\,\ud\tau.
\end{align}
Using
\begin{align}
        \int_{\mathbb S^{d-1}}x_i^2\,\dvol_{\mathbb S^{d-1}}
        =
        \frac{\omega_{d-1}}{d},
        \qquad
        \int_{\mathbb S^{d-1}}|\nabla_{\mathbb S^{d-1}}x_i|^2\,\dvol_{\mathbb S^{d-1}}
        =
        \frac{(d-1)\omega_{d-1}}{d},
\end{align}
and \(\delta^2=d-2\), we compute
\begin{align}
        \delta_U
        &=
        \frac{\omega_{d-1}}{d}
        \left[
        2\delta^d I_{d+1}
        +
        2\delta^{d-1}J_0
        -
        \delta^{d-2}
        \int_0^{\pi/2}
        \bigl(\cos^2t+d-1\bigr)\sin^{d-1}t\,\ud t
        -
        (d-1)\delta^{d-3}J_0
        \right]  \\
        &=
        \frac{\omega_{d-1}}{d}
        \left[
        \delta^{d-2}\bigl((2d-3)I_{d+1}-dI_{d-1}\bigr)
        +
        \delta^{d-3}(d-3)J_0
        \right]  \\
        &=
        \frac{\omega_{d-1}}{d}
        \left[
        \delta^{d-2}\frac{d(d-4)}{d+1}I_{d-1}
        +
        \delta^{d-3}(d-3)J_0
        \right]
        >0,
\end{align}
where in the last line we used the recurrence $I_{d+1}=\frac{d}{d+1}I_{d-1}$ and $d\geq 4.$ Next, we choose \(\veps>0\) so small that
\(3C_d\veps<\delta_U\). Then, using \eqref{eq:angular-error-estimate-new}, we get
\begin{align}
        \int |\nabla U_{\veps,i}|_{g_\veps}^2\,\ud\nu_\veps \leq
        \calN_U^0+C_d\veps
        =
        2\calD_U^0-\delta_U+C_d\veps<
        2(\calD_U^0-C_d\veps)\leq  2\int U_{\veps,i}^2\,\ud\nu_\veps.
\end{align}
Hence,
\begin{align}\label{eq:angular-strict-new}
        \cR_\veps(U_{\veps,i})<2
        \qquad(i=1,\dots,d).
\end{align}
We next define another test function coming from the radial direction along the cylindrical end. Denote
\begin{align}
        b_A(r)=
        \begin{cases}
        \cos(r/\delta),&0\leq r\leq B,\\
        -A(r-B),&r\geq B,
        \end{cases}
\end{align}
where \(A>0\) is chosen so that \(\int b_A\,\ud\nu_0=0\). This can be done since
\begin{align}
        \int b_A\,\ud\nu_0
        &=
        \omega_{d-1}
        \int_0^B
        \cos(r/\delta)\bigl(\delta\sin(r/\delta)\bigr)^{d-1}\,\ud r
        -
        A\omega_{d-1}\delta^{d-1}
        \int_0^\infty
        \ta e^{-\ta^2/2}\,\ud\ta\\
        &=
        \omega_{d-1}
        \int_0^B
        \cos(r/\delta)\bigl(\delta\sin(r/\delta)\bigr)^{d-1}\,\ud r
        -
        A\omega_{d-1}\delta^{d-1}.
\end{align}
The first term is strictly positive, since \(\cos(r/\delta)>0\) and
\(\sin(r/\delta)>0\) for \(0<r<B\). Hence there is a unique choice
\begin{align}
        A
        =
        \frac{
        \int_0^B
        \cos(r/\delta)\bigl(\delta\sin(r/\delta)\bigr)^{d-1}\,\ud r
        }{
        \delta^{d-1}
        }
        >0
\end{align}
for which \(\int b_A\,\ud\nu_0=0\). Our test function will be $B_\veps:=b_A-c_\veps,$ where we choose $c_\veps:=Z_\veps^{-1} \int b_A\,\ud\nu_\veps$ so that \(\int B_\veps\,\ud\nu_\veps=0\). Although \(b_A\) is generally not \(C^1\) at \(r=B\), it is continuous and piecewise smooth. Hence its weak derivative is the piecewise derivative
\begin{align}
        b_A'(r)=
        \begin{cases}
        -\delta^{-1}\sin(r/\delta),&0< r<B,\\
        -A,&r>B.
        \end{cases}
\end{align}
The function $b_A$ has at most linear growth and $b_A'$ is bounded. Since
$\nu_\veps$ has finite mass and has a Gaussian tail on the cylindrical end,
$b_A,B_\veps\in W^{1,2}(\mathbb R^d,g_\veps,\nu_\veps)$. Thus $B_\veps$ is an admissible test function, and all appearances of $B_\veps'=b_A'$ below are understood in the weak sense. We first note that
\(c_\veps=O_d(\veps)\). This is because \(\int b_A\,\ud\nu_0=0\),
\eqref{eq:end-density-error}, and
\begin{align}
        &\left|\int b_A\,\ud\nu_\veps\right|
        +
        \left|Z_\veps-\int1\,\ud\nu_0\right|
        =
        \left|\int b_A\,\ud\nu_\veps-\int b_A\,\ud\nu_0\right|
        +
        \left|\int1\,\ud\nu_\veps-\int1\,\ud\nu_0\right| \notag\\
        &\leq
        \omega_{d-1}
        \int_0^\infty
        (1+|b_A(r)|)\left|e^{-V_\veps(r)}\rho_{\delta,\veps}(r)^{d-1}
        -
        e^{-V_0(r)}\rho_0(r)^{d-1}\right|\,\ud r \notag\\
        &\leq
        C_d\int_{I_\veps}1\,\ud r
        +
        C_d\int_{B+\veps}^\infty
        (1+r-B)\delta^{d-1}
        \left|e^{-V_\veps(r)}-e^{-V_0(r)}\right|
        \,\ud r \notag\\
        &\leq
        C_d\veps
        +
        C_d\veps
        \int_{B+\veps}^\infty
        (1+r-B)^2e^{-(r-B)^2/2}\,\ud r
        \leq
        C_d\veps.
\end{align}
Moreover, choosing $\veps\leq B/4$, we have that $B/2\leq B-2\veps$ and therefore $V_\veps\equiv 0$ on $[0,B/2]$ which implies
\begin{align}
        Z_\veps
        =
        \int_{\bbR^d}1\,\ud\nu_\veps
        &\geq
        \omega_{d-1}
        \int_0^{B/2}
        e^{-V_\veps(r)}
        \rho_{\delta,\veps}(r)^{d-1}\,\ud r=
        \omega_{d-1}
        \int_0^{B/2}
        \bigl(\delta\sin(r/\delta)\bigr)^{d-1}\,\ud r
        =
        c_d>0.
\end{align}
Thus,
\begin{align}\label{eqn:c_eps-bound}
        |c_\veps|
        =
        \left|
        \frac{\int b_A\,\ud\nu_\veps}{Z_\veps}
        \right|
        \leq
        \frac{C_d\veps}{c_d}
        \leq
        C_d\veps.
\end{align}
Denote
\begin{align}
        \calD_{B_\veps}^0
        =
        \int B_\veps^2\,\ud\nu_0,
        \qquad
        \calN_{B_\veps}^0
        =
        \int |B_\veps'|^2\,\ud\nu_0.
\end{align}
Since \(B_\veps=b_A-c_\veps\), \(\int b_A\,\ud\nu_0=0\), and
\(B_\veps'=b_A'\), we have
\begin{align}
        \calD_{B_\veps}^0
        =
        \int b_A^2\,\ud\nu_0
        +
        c_\veps^2\int1\,\ud\nu_0,\quad \calN_{B_\veps}^0
        =
        \int |b_A'|^2\,\ud\nu_0.
\end{align}
We next claim that
\[
        \delta_b
        :=
        2\int b_A^2\,\ud\nu_0
        -
        \int |b_A'|^2\,\ud\nu_0
        >0.
\]
First, the condition \(\int b_A\,\ud\nu_0=0\) gives the explicit value
\[
        A
        =
        \frac{
        \int_0^B
        \cos(r/\delta)(\delta\sin(r/\delta))^{d-1}\,\ud r
        }{\delta^{d-1}}
        =
        \frac{\delta}{d}.
\]
% Indeed, after the change of variables \(t=r/\delta\),
% \[
%         \int_0^B
%         \cos(r/\delta)(\delta\sin(r/\delta))^{d-1}\,\ud r
%         =
%         \delta^d\int_0^{\pi/2}\cos t\,\sin^{d-1}t\,\ud t
%         =
%         \frac{\delta^d}{d}.
% \]
Also, by integration by parts,
\begin{align}
        \int_0^\infty \tau^2e^{-\tau^2/2}\,\ud\tau
        =
        \int_0^\infty e^{-\tau^2/2}\,\ud\tau
        =J_0.
\end{align}
Now using the same notation as in \eqref{eqn:notation} we get
\begin{align}
        \delta_b
        &=
        \omega_{d-1}
        \left[
        2\delta^d\int_0^{\pi/2}\cos^2t\,\sin^{d-1}t\,\ud t
        +
        2\delta^{d-1}A^2J_0
        -
        \delta^{d-2}I_{d+1}
        -
        \delta^{d-1}A^2J_0
        \right] \\
        &=
        \omega_{d-1}
        \left[
        \delta^{d-2}
        \left(
        2(d-2)(I_{d-1}-I_{d+1})-I_{d+1}
        \right)
        +
        \delta^{d-1}A^2J_0
        \right] \\
        &=
        \omega_{d-1}
        \left[
        \delta^{d-2}\frac{(d-4)}{d+1}I_{d-1}
        +
        \frac{\delta^{d+1}}{d^2}J_0
        \right]
        >0.
\end{align}
Here we used \(I_{d+1}=\frac{d}{d+1}I_{d-1}\), \(\delta^2=d-2\), \(A=\delta/d\), and  \(d\geq4\). Then
\begin{align}
        \delta_{B_\veps}^0
        &:=
        2\calD_{B_\veps}^0-\calN_{B_\veps}^0=
        \delta_b+2c_\veps^2\int1\,\ud\nu_0
        \geq
        \delta_b>0.
\end{align}
We next compare the \(\nu_\veps\)-mass and Dirichlet energy of \(B_\veps\) to these limiting
quantities. Since \(|c_\veps|\leq C_d\veps\), on the cylindrical end
\begin{align}
        |B_\veps|^2+|B_\veps'|^2
        \leq C_d(1+(r-B)^2).
\end{align}
Thus, recalling that \(I_\veps=[B-2\veps,B+\veps]\) and using \eqref{eq:end-density-error}, we get
\begin{align}\label{eq:Beps-error-estimate}
        &\left|
        \int B_\veps^2\,\ud\nu_\veps-\calD_{B_\veps}^0
        \right|
        +
        \left|
        \int |B_\veps'|^2\,\ud\nu_\veps-\calN_{B_\veps}^0
        \right| \notag\\
        &\leq
        C_d\int_{I_\veps}1\,\ud r
        +
        C_d\int_B^\infty
        \bigl(1+(r-B)^2\bigr)
        \left|e^{-V_\veps(r)}-e^{-V_0(r)}\right|
        \,\ud r \notag\\
        &\leq
        C_d\veps
        +
        C_d\veps
        \int_B^\infty
        \bigl(1+(r-B)^2\bigr)(1+r-B)e^{-(r-B)^2/2}
        \,\ud r
        \leq
        C_d\veps.
\end{align}
Moreover,
\begin{align}
        \int b_A^2\,\ud\nu_0
        &\geq
        \omega_{d-1}
        \delta^d
        \int_0^{\pi/2}\cos^2t\,\sin^{d-1}t\,\ud t  
        =:m_d>0.
\end{align}
Choose \(\veps>0\) so small that \(C_d\veps<\frac12\calD_{B_\veps}^0\), which is possible since $\calD_{B_\veps}^0\geq \int b^2_A \ud\nu_0\geq m_d>0$, and
\(3C_d\veps<\delta_b\). Then, using \eqref{eq:Beps-error-estimate} and
\(\delta_{B_\veps}^0\geq\delta_b\), we obtain
\begin{align}
        \int |\nabla B_\veps|_{g_\veps}^2\,\ud\nu_\veps
        &=
        \int |B_\veps'|^2\,\ud\nu_\veps
        \leq
        \calN_{B_\veps}^0+C_d\veps\\
        &=
        2\calD_{B_\veps}^0-\delta_{B_\veps}^0+C_d\veps
        <
        2(\calD_{B_\veps}^0-C_d\veps)\\
        &\leq
        2\int B_\veps^2\,\ud\nu_\veps.
\end{align}
Hence
\begin{align}\label{eq:radial-strict-new}
        \cR_\veps(B_\veps)<2.
\end{align}
Let
\begin{align}
        F_\veps
        =
        \spn\{1,B_\veps,U_{\veps,1},\dots,U_{\veps,d}\}.
\end{align}
We show that the functions spanning $F_\veps$ are pairwise \(L^2(\nu_\veps)\)-orthogonal. Since
\begin{align}
        \ud\nu_\veps
        =
        e^{-V_\veps(r)}\rho_{\delta,\veps}(r)^{d-1}\,\ud r\,\dvol_{\bbS^{d-1}},
\end{align}
we have
\begin{align}
        \int_{\bbR^d} B_\veps\,\ud\nu_\veps
        &=0,\\
        \int_{\bbR^d} U_{\veps,i}\,\ud\nu_\veps
        &=
        \left(\int_0^\infty
        a_\veps(r)e^{-V_\veps(r)}\rho_{\delta,\veps}(r)^{d-1}\,\ud r
        \right)
        \left(\int_{\bbS^{d-1}}x_i\,\dvol_{\bbS^{d-1}}\right)
        =
        0,\\
        \int_{\bbR^d} B_\veps U_{\veps,i}\,\ud\nu_\veps
        &=
        \left(\int_0^\infty
        B_\veps(r)a_\veps(r)e^{-V_\veps(r)}
        \rho_{\delta,\veps}(r)^{d-1}\,\ud r
        \right)
        \left(\int_{\bbS^{d-1}}x_i\,\dvol_{\bbS^{d-1}}\right)
        =
        0,\\
        \int_{\bbR^d} U_{\veps,i}U_{\veps,j}\,\ud\nu_\veps
        &=
        \left(\int_0^\infty
        a_\veps(r)^2e^{-V_\veps(r)}
        \rho_{\delta,\veps}(r)^{d-1}\,\ud r
        \right)
        \left(\int_{\bbS^{d-1}}x_ix_j\,\dvol_{\bbS^{d-1}}\right)
        =
        0
        \quad(i\neq j).
\end{align}
Thus the functions in the spanning set are \(L^2(\nu_\veps)\)-orthogonal. In
particular they are linearly independent, and hence \(\dim F_\veps=d+2\). We next show orthogonality with respect to the Dirichlet form. To this end, we first compute
the gradients,
\begin{align}
        \nabla B_\veps=B_\veps'\partial_r,
        \qquad
        \nabla U_{\veps,i}
        =
        a_\veps'x_i\partial_r
        +
        \frac{a_\veps}{\rho_{\delta,\veps}^2}
        \nabla_{\bbS^{d-1}}x_i.
\end{align}
As a result, we get
\begin{align}
        \int_{\bbR^d}\la\nabla B_\veps,\nabla U_{\veps,i}\ra_{g_\veps}\,\ud\nu_\veps
        &=
        \left(\int_0^\infty
        B_\veps'(r)a_\veps'(r)e^{-V_\veps(r)}
        \rho_{\delta,\veps}(r)^{d-1}\,\ud r
        \right)
        \left(\int_{\bbS^{d-1}}x_i\,\dvol_{\bbS^{d-1}}\right)
        =
        0,
\end{align}
for all \(1\leq i\leq d\). Moreover,
\begin{align}
        \int_{\bbR^d}\la\nabla U_{\veps,i},\nabla U_{\veps,j}\ra_{g_\veps}\,\ud\nu_\veps
        &=
        \left(\int_0^\infty
        (a_\veps')^2e^{-V_\veps}
        \rho_{\delta,\veps}^{d-1}\,\ud r
        \right)
        \left(\int_{\bbS^{d-1}}x_ix_j\,\dvol_{\bbS^{d-1}}\right)\\
        &\quad+
        \left(\int_0^\infty
        a_\veps^2e^{-V_\veps}
        \rho_{\delta,\veps}^{d-3}\,\ud r
        \right)
        \left(\int_{\bbS^{d-1}}
        \la\nabla_{\bbS^{d-1}}x_i,\nabla_{\bbS^{d-1}}x_j\ra
        \,\dvol_{\bbS^{d-1}}\right)\\
        &=
        0,
\end{align}
for all \(1\leq i,j\leq d\), \(i\neq j\). Here the last equality follows because
\begin{align}
        \int_{\bbS^{d-1}}x_ix_j\,\dvol_{\bbS^{d-1}}=0
        \quad(i\neq j),
\end{align}
and, using \(-\Delta_{\bbS^{d-1}}x_j=(d-1)x_j\),
\begin{align}
        \int_{\bbS^{d-1}}
        \la\nabla_{\bbS^{d-1}}x_i,\nabla_{\bbS^{d-1}}x_j\ra
        \,\dvol_{\bbS^{d-1}}
        =
        (d-1)\int_{\bbS^{d-1}}x_ix_j\,\dvol_{\bbS^{d-1}}
        =
        0.
\end{align}
Thus the spanning functions are pairwise orthogonal both in \(L^2(\nu_\veps)\)
and with respect to the Dirichlet form. Then, for every nonzero \(u\in F_\veps\),
\begin{align}
        \cR_\veps(u)
        \leq
        \max\{\cR_\veps(B_\veps),
              \cR_\veps(U_{\veps,1}),\dots,\cR_\veps(U_{\veps,d})\}
        <2.
\end{align}
By the min-max characterization of the variational eigenvalues,
\begin{align}
        \lambda_{d+2}(\bbR^d,g_\veps,\mu_\veps)
        \leq
        \sup_{0\neq u\in F_\veps}\cR_\veps(u)
        <2.
\end{align}
Recall that \eqref{eqn:gaussian-spectrum} implies that
\(\lambda_{d+2}(\bbR^d,|\cdot|,\gamma^d)=2\). Thus
\begin{align}
        \lambda_{d+2}(\bbR^d,g_\veps,\mu_\veps)
        <
        \lambda_{d+2}(\bbR^d,|\cdot|,\gamma^d),
\end{align}
and the theorem follows.
\end{proof}
\begin{cor}\label{cor:gaussian}
Let \((\R^d,g,\mu)\) be the weighted manifold constructed in Theorem
\ref{thm:gaussian-main}, which satisfies \(\CD(1,\infty)\). Then there does not exist a $1$-Lipschitz map
\[
        T:(\R^d,|\cdot|,\gamma^d)\to (\R^d,g,\mu)
\]
pushing forward $\gamma^d$ onto $\mu$ up to a finite constant.
\end{cor}
\begin{proof}
If such a map existed, then Theorem~\ref{thm:contraction} would imply that for all $k\geq 1$ we have
\begin{align}
   \lam_k(\R^d,g,\mu) \geq  \lam_k(\R^d,|\cdot|,\gamma^d).
\end{align}
This contradicts \eqref{eqn:index-gaussian-K}.
\end{proof}

\section{Proof of Theorem~\ref{thm:gaussian-split-main}}\label{sec:proof-of-gaussian-split}
The proof is based on a volume-preserving perturbation of the round metric on
$\bbS^{d-1}$. We first use this perturbation to construct an explicit
finite-dimensional space of ordinary polynomials with small Rayleigh quotient
on a weighted metric cone. We then smooth the vertex of the cone by a
trace-free interpolation.  For a metric $g$ and a smooth function $V$, we write
\begin{align}
        \cR_{g,V}(u)
        :=
        \frac{\int |\nabla u|_g^2e^{-V}\dvol_g}
        {\int u^2e^{-V}\dvol_g}.
\end{align}
When $V=0$, we simply write $\cR_g(u)$. Let $g_{\can}$ denote the
canonical round metric on $\bbS^n$. For a metric $\bar{g}$ on $\bbS^n$ and a number
$s>0$, define
\begin{align}\label{eq:gaussian-split-Ahs}
        A(\bar{g},s)
        :=
        \max\left\{
        1,
        \sup_{\substack{\omega\in\bbS^n\\0\neq\xi\in T_\omega^*\bbS^n}}
        \frac{s^{-2}|\xi|_{\bar{g}}^2}{|\xi|_{g_\can}^2}
        \right\}.
\end{align}
This constant arises naturally when comparing the Dirichlet energy of functions written over cones in either the spherical metric or the perturbed metric $\bar{g}.$ For $\vartheta\in\bbR$, define the rotation
$R_\vartheta:\bbS^n\to\bbS^n$ by $R_\vartheta(z,y):=(e^{\iii\vartheta}z,y)$, where we use $\bbS^n\subset \bbC\times \R^{n-1}$ when $n\geq 4.$

We begin with the required perturbation of the spherical link.
\begin{lem}\label{lem:gaussian-split-link}
Let $n\geq4$. There exist an integer $q\geq2$, parameters $0<s<\rho<1,$ a smooth metric $\bar g$ on $\bbS^n$, and a smooth family of metrics
$\{g_\ta\}_{\ta\in[0,1]}$ on $\bbS^n$ satisfying $R_\vartheta^*g_\ta=g_\ta$ for every $\vartheta\in\bbR$ and $\ta\in[0,1]$ on $\bbS^n$ such that ${g}_0=g_{\can},$ ${g}_1=\bar{g}$ and for every $\ta\in [0,1]$ we have $\dvol_{{g}_\ta}=\dvol_{g_\can}$, $\operatorname{tr}_{g_\ta}(\partial_\ta g_\ta)=0$ and $\Ric_{{g}_\ta}\geq(n-1)\rho^2{g}_\ta$. Moreover, if
\begin{align}
        f_q(z,y):=\Re(z^q),
        \qquad
        (z,y)\in\bbS^n\subset\bbC\times\bbR^{n-1},
\end{align}
then
\begin{align}\label{eq:gaussian-split-link-Rayleigh}
                \cR_{\bar g}(f_q)<s^2q(q+n-1),
\end{align}
and
\begin{align}\label{eq:gaussian-split-link-comparison}
        A(\bar{g},s)(q-1)<q.
\end{align}
\end{lem}

\begin{proof}
Using the notation from Section~\ref{sec:sphere-thm-main}, we recall that 
\begin{align}
        g_\can 
        =
        \ud t^2+\sin^2t\,\ud\theta^2
        +\cos^2t\,g_{\bbS^{n-2}}.
\end{align}
For $0\leq \ta\leq1$ and $\veps>0$, define
\begin{align}\label{eq:gaussian-split-link-metric}
        g_\ta
        :=
        e^{-2 \ta \veps\sin^2t/(n-1)}
        \left(
        \ud t^2+\cos^2t\,g_{\bbS^{n-2}}
        \right)
        +
        e^{2 \ta \veps\sin^2t}\sin^2t\,\ud\theta^2.
\end{align}
Write $z=x_1+\iii x_2$, and set
\begin{align}
        \upsilon
        &:=
        |z|^2,
        \qquad
        \beta_0
        :=
        x_1\,\ud x_2-x_2\,\ud x_1.
\end{align}
Define
\begin{align}
        b_\ta(\upsilon)
        &:=
        \begin{cases}
        \displaystyle
        \frac{
        e^{2\ta\veps\upsilon}
        -
        e^{-\frac{2\ta\veps}{n-1}\upsilon}
        }{\upsilon},
        &
        \upsilon>0,
        \\[2mm]
        \displaystyle
        \frac{2n\ta\veps}{n-1},
        &
        \upsilon=0.
        \end{cases}
\end{align}
The function $b_\ta$ is smooth in $(\ta,\upsilon)$ and nonnegative.
Since
\begin{align}
        |z|^2
        =
        \sin^2t,
        \quad
        \beta_0
        =|z|^2\,\ud\theta
\end{align}
on the set $U=\{(z,y)\in \bS^n:z\neq 0,y\neq 0\}$ (equivalently when $0<t<\pi/2$), we may rewrite the metric on this set as
\begin{align}
        g_\ta
        =
        e^{-\frac{2\ta\veps}{n-1}|z|^2}g_{\can}
        +
        b_\ta(|z|^2)\beta_0\otimes\beta_0.
\end{align}
Hence $\{g_\ta\}_{\ta\in[0,1]}$ extends to a smooth family of
Riemannian metrics on all of $\bbS^n$. Since the coefficients in \eqref{eq:gaussian-split-link-metric} are independent of $\theta$, we have $R_\vartheta^*g_\ta=g_\ta$ for every $\vartheta\in\bbR$ and $\ta\in[0,1]$. Furthermore, $g_0= g_\can$ and
% We first verify that $g_\ta$ extends smoothly across the two singular orbits.
% Near $t=0$, introduce the radial arclength coordinate, 
% \begin{align}
%         r(t)
%         :=
%         \int_0^t
%         e^{-\tau \veps\sin^2\sigma/(n-1)}\,\ud\sigma.
% \end{align}
% Then
% \begin{align}
%         r=t+O(t^3),
% \end{align}
% and the warping function of the collapsing circle satisfies
% \begin{align}
%         e^{u\veps\sin^2t}\sin t=t+O(t^3).
% \end{align}
% Expressed as a function of $r$, this warping function is smooth and odd, with
% value zero and first derivative one at $r=0$. The remaining coefficients are
% smooth even functions of $r$. Thus the circle closes smoothly at $t=0$.
% Near $t=\pi/2$, the collapsing block
% \begin{align}
%         \ud t^2+\cos^2t\,g_{\bbS^{n-2}}
% \end{align}
% is multiplied by one smooth positive even function of $\pi/2-t$, while the
% coefficient of $\ud\theta^2$ is also a smooth even function. Thus the
% $\bbS^{n-2}$-factor closes smoothly at $t=\pi/2$.
this family of metrics is volume preserving since
\begin{align}
        \dvol_{g_\ta}
        =
        \left(
        e^{-\frac{2\ta\veps\sin^2t}{n-1}}
        \right)^{\frac{n-1}{2}}
        \left(
        e^{2\ta\veps\sin^2t}
        \right)^{\frac12}
        \dvol_{g_0}
        =
        e^{-\ta\veps\sin^2t}
        e^{\ta\veps\sin^2t}
        \dvol_{g_0}
        =
        \dvol_{g_0}.
        \label{eq:gaussian-split-volume-preserving}
\end{align}
Since
\begin{align}
        0=\partial_\tau \dvol_{g_\ta}=\frac{1}{2}\tr_{g_\ta} (\partial_\ta g_\ta) \dvol_{g_\ta}
        =
        \frac12\operatorname{tr}
        \bigl(g_\ta^{-1}\partial_\ta g_\ta\bigr)\dvol_{g_\ta},
\end{align}
we also see that $\tr(g_\ta^{-1}\partial_\ta g_\ta)=0.$ We next estimate the Ricci curvature of the metric $g_\ta$. Recall that on the
coordinate set under consideration, a point of $\bbS^n$ is written as
\begin{align}
        (z,y)
        =
        \bigl(e^{\iii\theta}\sin t,(\cos t)\omega\bigr),
        \qquad
        t\in(0,\pi/2),
        \quad
        \theta\in [0,2\pi),
        \quad
        \omega\in\bbS^{n-2}.
\end{align}
Thus $\theta$ is the angular coordinate on the $\bbS^1$-factor, while
$\omega$ denotes a point of the $\bbS^{n-2}$-factor.

Fix $\ta\in[0,1]$. We define the arclength coordinate in the $t$-direction by
\begin{align}\label{eq:gaussian-split-link-arclength}
        r(t)
        :=
        \int_0^t
        e^{-\frac{\ta\veps}{n-1}\sin^2\sigma}\,\ud\sigma.
\end{align}
Then
\begin{align}\label{eq:gaussian-split-r-t-derivative}
        \ud r 
        =
        e^{-\frac{\ta\veps}{n-1}\sin^2t}
        \ud t.
\end{align}
Define the functions appearing in the warped product by
\begin{align}\label{eq:gaussian-split-link-warping-functions}
        a_\ta(r)
        :=
        e^{\ta\veps\sin^2t}\sin t,\quad 
        c_\ta(r)
        :=
        e^{-\frac{\ta\veps}{n-1}\sin^2t}\cos t,
\end{align}
Then \eqref{eq:gaussian-split-link-metric} is simply $ g_\ta =\ud r^2+ a_\ta(r)^2\,\ud\theta^2+c_\ta(r)^2g_{\bbS^{n-2}}.$  Let $\{v_j\}_{j=1}^{n-2}$ be a local orthonormal frame with respect to the round metric on $\bbS^{n-2}$. Then
\begin{align}
        e_r
        :=
        \partial_r,\quad
        e_\theta
        :=
        \frac1{a_\ta}\partial_\theta,\quad 
        e_j
        :=
        \frac1{c_\ta}v_j,
        \qquad
        1\leq j\leq n-2,
\end{align}
form a local orthonormal frame with respect to the metric $g_\ta$. In these coordinates, we define the following components of the Ricci curvature as
\begin{align}\label{eq:gaussian-split-link-Ricci-eigenvalue-definitions}
        R_r
        :=
        \Ric_{g_\ta}(e_r,e_r),\quad 
        R_\theta
        :=
        \Ric_{g_\ta}(e_\theta,e_\theta),\quad 
        R_S
        :=
        \Ric_{g_\ta}(e_j,e_j),
        \quad
        1\leq j\leq n-2,
\end{align}
where we note that $R_S$ is independent of the index $1\leq j\leq n-2$. We now apply Lemma \ref{lem:multiply-warped-ricci} to estimate the Ricci curvature of each component in the above display. First, we record some derivatives of the warping functions. Denote $\partial_r a_\ta =a'_\ta$, $\partial_r c_\ta=c'_\ta$ and so on for the second derivative. Then, we have 
\begin{align}
        \frac{a_\ta'}{a_\ta}
        &=
        e^{\frac{\ta\veps}{n-1}\sin^2t}
        \left(
        \cot t
        +
        2\ta\veps\sin t\cos t
        \right),\quad 
        \frac{c_\ta'}{c_\ta}
        =
        -
        e^{\frac{\ta\veps}{n-1}\sin^2t}
        \left(
        \tan t
        +
        \frac{2\ta\veps}{n-1}\sin t\cos t
        \right), \label{eq:gaussian-split-link-first-log-derivatives}\\
        \frac{a_\ta'c_\ta'}{a_\ta c_\ta}
        &=
        -
        e^{\frac{2\ta\veps}{n-1}\sin^2t}
        \left(
        1
        +
        2\ta\veps\sin^2t
        +
        \frac{2\ta\veps}{n-1}\cos^2t
        +
        \frac{4(\ta\veps)^2}{n-1}
        \sin^2t\cos^2t
        \right).\label{eq:gaussian-split-link-cross-log-derivative}
\end{align}
By differentiating \eqref{eq:gaussian-split-link-first-log-derivatives} we obtain
\begin{align}\label{eq:gaussian-split-link-second-log-derivatives}
        \frac{a_\ta''}{a_\ta}
        &=
        e^{\frac{2\ta\veps}{n-1}\sin^2t}
        \left(
        -1
        +
        \frac{\ta\veps}{n-1}
        \left(
        6n-4-(8n-6)\sin^2t
        \right)
        +
        \frac{4n(\ta\veps)^2}{n-1}
        \sin^2t\cos^2t
        \right),\\
        \frac{c_\ta''}{c_\ta}
        &=
        e^{\frac{2\ta\veps}{n-1}\sin^2t}
        \left(
        -1
        +
        \frac{2\ta\veps}{n-1}
        \left(
        3\sin^2t-1
        \right)
        \right).
\end{align}
To estimate some numerical constants arising in \eqref{eq:gaussian-split-link-first-log-derivatives},\eqref{eq:gaussian-split-link-cross-log-derivative} and \eqref{eq:gaussian-split-link-second-log-derivatives} we define $C_n:=n-1+\frac{2n}{n-1}.$ Since $n\geq4$, we have $\frac{4n}{n-1}<C_n<2(n-1).$ We choose $\veps>0$ sufficiently small that
\begin{align}\label{eq:gaussian-split-epsilon-first}
        \frac{4n}{n-1}
        +
        \frac{n}{n-1}\veps
        &\leq
        C_n,
        \quad
        C_n\veps<n-1,
\end{align}
and set $\rho^2:= 1-\frac{C_n\veps}{n-1}.$  In particular, $0<\rho<1$. Since $0\leq\ta\leq1$, we have $0\leq\ta\veps\leq\veps$, and therefore
\begin{align}\label{eq:gaussian-split-Ricci-error-bound}
        \frac{4n}{n-1}\ta\veps
        +
        \frac{n}{n-1}(\ta\veps)^2
        =
        \ta\veps
        \left(
        \frac{4n}{n-1}
        +
        \frac{n}{n-1}\ta\veps
        \right)\leq
        \ta\veps
        \left(
        \frac{4n}{n-1}
        +
        \frac{n}{n-1}\veps
        \right)\leq
        C_n\ta\veps
        \leq
        C_n\veps.
\end{align}
Moreover, since $C_n>2$, we have $ 2\ta\veps \leq C_n\veps.$  Now applying Lemma~\ref{lem:multiply-warped-ricci} and
\eqref{eq:gaussian-split-link-second-log-derivatives}, we get
\begin{align}
        R_r
        &=
        -\frac{a_\ta''}{a_\ta}
        -(n-2)\frac{c_\ta''}{c_\ta}\\
        &=
        e^{\frac{2\ta\veps}{n-1}\sin^2t}
        \left(
        n-1
        +
        \frac{\ta\veps}{n-1}
        \left(
        (2n+6)\sin^2t-4n
        \right)
        -
        \frac{4n(\ta\veps)^2}{n-1}
        \sin^2t\cos^2t
        \right)\\
        &\geq
        e^{\frac{2\ta\veps}{n-1}\sin^2t}
        \left(
        n-1
        -
        \frac{4n}{n-1}\ta\veps
        -
        \frac{n}{n-1}(\ta\veps)^2
        \right)\\
        &\geq
        n-1-C_n\veps = (n-1)\rho^2.
        \label{eq:gaussian-split-link-Rr}
\end{align}
Similarly, using Lemma~\ref{lem:multiply-warped-ricci},
\eqref{eq:gaussian-split-link-second-log-derivatives}, and
\eqref{eq:gaussian-split-link-cross-log-derivative}, we obtain
\begin{align}
        R_\theta
        &=
        -\frac{a_\ta''}{a_\ta}
        -(n-2)\frac{a_\ta'c_\ta'}{a_\ta c_\ta}\\
        &=
        e^{\frac{2\ta\veps}{n-1}\sin^2t}
        \left(
        n-1
        +
        \frac{\ta\veps}{n-1}
        \left(
        2(n^2+1)\sin^2t-4n
        \right)
        -
        \frac{8(\ta\veps)^2}{n-1}
        \sin^2t\cos^2t
        \right)\\
        &\geq
        e^{\frac{2\ta\veps}{n-1}\sin^2t}
        \left(
        n-1
        -
        \frac{4n}{n-1}\ta\veps
        -
        \frac{2}{n-1}(\ta\veps)^2
        \right)\\
        &\geq
        e^{\frac{2\ta\veps}{n-1}\sin^2t}
        \left(
        n-1
        -
        \frac{4n}{n-1}\ta\veps
        -
        \frac{n}{n-1}(\ta\veps)^2
        \right)\\
        &\geq
        n-1-C_n\veps=
        (n-1)\rho^2.
        \label{eq:gaussian-split-link-Rtheta}
\end{align}
Finally, since $\Ric_{\bbS^{n-2}}= (n-3)g_{\bbS^{n-2}},$ Lemma~\ref{lem:multiply-warped-ricci},
\eqref{eq:gaussian-split-link-first-log-derivatives}, and
\eqref{eq:gaussian-split-link-cross-log-derivative} imply that
\begin{align}
        R_S
        &=
        \frac{n-3}{c_\ta^2}
        -
        \frac{c_\ta''}{c_\ta}
        -
        (n-3)
        \left(
        \frac{c_\ta'}{c_\ta}
        \right)^2
        -
        \frac{a_\ta'c_\ta'}{a_\ta c_\ta}\\
        &=
        e^{\frac{2\ta\veps}{n-1}\sin^2t}
        \left(
        n-1
        +
        \ta\veps
        \left(
        \frac4{n-1}-2\sin^2t
        \right)
        +
        \frac{8(\ta\veps)^2}{(n-1)^2}
        \sin^2t\cos^2t
        \right)\\
        &\geq
        e^{\frac{2\ta\veps}{n-1}\sin^2t}
        \left(
        n-1-2\ta\veps
        \right)\\
        &\geq
        n-1-C_n\veps= (n-1)\rho^2.
        \label{eq:gaussian-split-link-RS}
\end{align}
Since $g_\ta$ is smooth on $\bbS^n$, the preceding Ricci curvature
inequalities extend to the two singular orbits by continuity. Hence
\begin{align}
        \Ric_{g_\ta}
        \geq
        (n-1)\rho^2g_\ta
\end{align}
on all of $\bbS^n$. Note that when $\ta=0$, then $g_\ta=g_\can$ and the above expressions reduce to
\begin{align}
        R_r=R_\theta=R_S=n-1
\end{align}
as expected. 

Set $\bar g:=g_1$ and
$f_q(z,y):=\Re(z^q)=\sin^qt\cos(q\theta)$. For every $j\geq0$, define
\begin{align}
 Q_j:=\int_{\bbS^n}\sin^{2j}t\,\dvol_{g_\can},\quad D_n:=2+\frac{2\eee}{(n-1)^2}.
\end{align}
Since $\frac{C_n}{n-1}<2$, we may choose an integer $q\geq2$ sufficiently large such that
\[
        \eta_{n,q}
        :=
        2q^2
        \left(
        1-\frac1{2q+n+1}
        \right)
        -
        \frac{C_n}{n-1}q(q+n-1)
        >
        0.
\]
Fix this value of $q$. For every sufficiently small $\veps>0$ satisfying all
the preceding conditions, we may assume that
$D_nq^2\veps<\frac12\eta_{n,q}$. Define
\begin{align}\label{eq:gaussian-split-s-definition}
        s^2
        :=
        \rho^2
        -
        \frac{\eta_{n,q}}{4q(q+n-1)}\veps=
        1-
        \left(
        \frac{C_n}{n-1}
        +
        \frac{\eta_{n,q}}{4q(q+n-1)}
        \right)\veps.
\end{align}
After decreasing $\veps>0$ if necessary, the right-hand side is positive.
Moreover, since $\eta_{n,q}>0$, we have $0<s<\rho$.

We now estimate the Rayleigh quotient \eqref{eq:gaussian-split-link-Rayleigh}. To this end, first observe that 
\begin{align}\label{eqn:rayleigh-estimate-helper}
  \int_{\bbS^n}f_q^2\,\dvol_{\bar g} &= \frac12Q_q  \\
  \int_{\bbS^n}|\nabla f_q|_{\bar g}^2\,\dvol_{\bar g}
        &=
        \frac{q^2}{2}
        \int_{\bbS^n}
        \sin^{2q-2}t
        \left(
        \cos^2t\,
        e^{\frac{2\veps}{n-1}\sin^2t}
        +
        e^{-2\veps\sin^2t}
        \right)
        \dvol_{g_\can}\\
        &\leq  \frac{q^2}{2}\int_{\bbS^n} \sin^{2q-2} t \left(2-\sin^2t
        +
        2\veps\sin^2t
        \left(
        \frac{\cos^2t}{n-1}-1
        \right)
        +
        D_n\veps^2\sin^4t\right) \dvol_{g_\can}\\
        &= \frac{q^2}{2}
        \left[
        2Q_{q-1}-Q_q
        +
        2\veps
        \left(
        \frac{Q_q-Q_{q+1}}{n-1}-Q_q
        \right)
        +
        D_n\veps^2Q_{q+1}
        \right],
\end{align}
where in the inequality above we used the pointwise inequality
$\eee^y\leq1+y+\frac \eee 2y^2$ and $\eee^{-y}\leq1-y+\frac12y^2$ for $0\leq y\leq1$, together with
$0<\veps\leq1/2$ and $D_n=2+\frac{2\eee}{(n-1)^2}$. Also, in the last equality above we used
\[
        \partial_tf_q
        =
        q\sin^{q-1}t\cos t\cos(q\theta),
        \qquad
        \partial_\theta f_q
        =
        -q\sin^qt\sin(q\theta),
\]
and
\[
        \int_0^{2\pi}\cos^2(q\theta)\,\ud\theta
        =
        \int_0^{2\pi}\sin^2(q\theta)\,\ud\theta
        =
        \pi.
\]
Dividing the preceding numerator estimate by
$\frac12Q_q$ gives
\begin{align}\label{eqn:R-estimate}
        \cR_{\bar g}(f_q)
        &\leq
        \frac{q^2}{Q_q}
        \left[
        2Q_{q-1}-Q_q
        +
        2\veps
        \left(
        \frac{Q_q-Q_{q+1}}{n-1}-Q_q
        \right)
        +
        D_n\veps^2Q_{q+1}
        \right]
        \notag\\
        &\leq
        q(q+n-1)
        -
        2\veps q^2
        \left(
        1-\frac1{2q+n+1}
        \right)
        +
        D_nq^2\veps^2\\
        &=
        \rho^2q(q+n-1)
        -
        \eta_{n,q}\veps
        +
        D_nq^2\veps^2\\
        &<
        \rho^2q(q+n-1)
        -
        \frac12\eta_{n,q}\veps\\
        &<
        \rho^2q(q+n-1)
        -
        \frac14\eta_{n,q}\veps=
        s^2q(q+n-1),
\end{align}

For the third line in the Rayleigh quotient estimate, we used
\begin{align}
        2Q_{q-1}-Q_q
        =
        \frac{q+n-1}{q}Q_q, \quad 
        \frac{Q_q-Q_{q+1}}{(n-1)Q_q}
        =
        \frac1{2q+n+1},\quad 
        Q_{q+1}
        \leq
        Q_q.
\end{align}
The first two identities in the above display follow from
$\frac{Q_j}{Q_{j-1}}=\frac{2j}{2j+n-1}$, while the last inequality in the above display follows
from $0\leq\sin^2t\leq1$. We next estimate
\begin{align}
        A(\bar g,s)(q-1)
        &=
        \max\left\{
        1,
        \sup_{\substack{\omega\in\bbS^n\\
        0\neq\xi\in T_\omega^*\bbS^n}}
        \frac{s^{-2}|\xi|_{\bar g}^2}
        {|\xi|_{g_\can}^2}
        \right\}
        (q-1)
        \leq
        \max\left\{
        1,
        s^{-2}e^{2\veps/(n-1)}
        \right\}
        (q-1)
        <
        q.
\end{align}
To justify the preceding estimate, write
\begin{align}
        \xi
        =
        \xi_t\,\ud t+\xi_\theta\,\ud\theta+\xi_S,
        \qquad
        \xi_S\in T_\omega^*\bS^{n-2}.
\end{align}
Then since  $\bar g
        =
        e^{-\frac{2\veps}{n-1}\sin^2t}
        \left(
        \ud t^2+\cos^2t\,g_{\bbS^{n-2}}
        \right)
        +
        e^{2\veps\sin^2t}
        \sin^2t\,\ud\theta^2,$
we have
\begin{align}
        |\xi|_{\bar g}^2
        =
        e^{\frac{2\veps}{n-1}\sin^2t}
        \left(
        \xi_t^2+
        \frac{|\xi_S|_{g_{\bbS^{n-2}}}^2}{\cos^2t}
        \right)
        +
        e^{-2\veps\sin^2t}
        \frac{\xi_\theta^2}{\sin^2t}
        &\leq
        e^{2\veps/(n-1)}
        \left(
        \xi_t^2+
        \frac{|\xi_S|_{g_{\bbS^{n-2}}}^2}{\cos^2t}
        +
        \frac{\xi_\theta^2}{\sin^2t}
        \right)
        \notag\\
        &=
        e^{2\veps/(n-1)}
        |\xi|_{g_\can}^2.
\end{align}
By \eqref{eq:gaussian-split-s-definition}, $s\to1$ as
$\veps\to 0 $. Therefore, as $\veps \to 0$
\begin{align}
        \max\left\{
        1,
        s^{-2}e^{2\veps/(n-1)}
        \right\}
        (q-1)
        \rightarrow
        q-1
        <
        q,
\end{align}
which proves
\begin{align}
        A(\bar g,s)(q-1)
        <
        q
\end{align}
for all sufficiently small $\veps>0$.
\end{proof}
Now, we will use the metric $\bar{g}$ constructed in the previous lemma and construct a cone over that metric with a suitable potential. The key point is to obtain a new weighted space with a suitable spectral property that will be needed for our counterexample. 
\begin{lem}\label{lem:gaussian-split-polynomial-cone}
Let $n\geq1$ and $q\geq2$. Let $g_0=g_\can$ be the canonical round metric
on $\bbS^n$, and let $\bar g$ be a metric on $\bbS^n$ satisfying
\begin{align}
        R_\vartheta^*\bar g
        =
        \bar g
        \qquad
        \text{for every }\vartheta\in\bbR,
        \qquad
        \dvol_{\bar g}
        =
        \dvol_{g_0}.
\end{align} 
Suppose that $s>0$ satisfies
\begin{align}\label{eq:gaussian-split-polynomial-assumptions}
        A(\bar g,s)(q-1)<q,
        \qquad
        \cR_{\bar g}(f_q)<s^2q(q+n-1),
\end{align}
where $f_q(z,y)=\Re(z^q)$. Then there exist $\alpha>1$ and a vector space
$E_C$ of ordinary polynomials on $\bbR^{n+1}$ such that, upon identifying
$\bbR^{n+1}\setminus\{0\}$ with $(0,\infty)\times\bbS^n$ by $x=r\omega$, we
have
\begin{align}\label{eq:gaussian-split-polynomial-dimension}
        \dim E_C=N_q+1,
        \qquad
        N_q:=\binom{n+q}{n+1},
\end{align}
and
\begin{align}\label{eq:gaussian-split-polynomial-Rayleigh}
        \sup_{0\neq p\in E_C}
        \cR_{g_C,V_C}(p)<q,
\end{align}
where
\begin{align}\label{eq:gaussian-split-polynomial-cone}
        C=(0,\infty)\times\bbS^n,
        \qquad
        g_C=\ud r^2+s^2r^2\bar g,
        \qquad
        V_C=\frac{\alpha r^2}{2}.
\end{align}
Moreover,
\begin{align}\label{eq:gaussian-split-Gaussian-index}
        \lam_{N_q+1}
        (\bbR^{n+1},|\cdot|,\gamma^{n+1})=q.
\end{align}
\end{lem}

\begin{proof}
Let
\begin{align}
        \cP_{<q}
        :=
        \left\{
        p:\bbR^{n+1}\to\bbR:
        p\text{ is a polynomial of degree strictly less than }q
        \right\}.
\end{align}
Since the number of monomials in $n+1$ variables of degree $j$ is
$\binom{n+j}{j}$, we have
\begin{align}\label{eq:gaussian-split-polynomial-count}
        \dim\cP_{<q}
        =
        \sum_{j=0}^{q-1}\binom{n+j}{j}
        =
        \binom{n+q}{n+1}
        =
        N_q.
\end{align}

We first estimate the Gaussian Rayleigh quotient on $\cP_{<q}$. Let
$L_0:=-\Delta+\la x,\nabla\cdot\ra$ denote the Ornstein--Uhlenbeck operator. The multivariable Hermite polynomials
\begin{align}
        H_\beta(x)
        :=
        \prod_{j=1}^{n+1}H_{\beta_j}(x_j),
        \qquad
        \beta\in\bbN_0^{n+1},
\end{align}
form an orthogonal basis of $L^2(\gamma^{n+1})$ and satisfy
$L_0H_\beta=|\beta|H_\beta$ (cf. Section~2.1 in
\cite{milman2018spectral}). Moreover, the Hermite polynomials with
$|\beta|<q$ span $\cP_{<q}$. Therefore, if
$p=\sum_{|\beta|<q}c_\beta H_\beta$, then
\begin{align}
        \int_{\bbR^{n+1}}|\nabla p|^2\,\ud\gamma^{n+1}
        =
        \int_{\bbR^{n+1}}pL_0p\,\ud\gamma^{n+1}
        &=
        \sum_{|\beta|<q}
        |\beta||c_\beta|^2
        \|H_\beta\|_{L^2(\gamma^{n+1})}^2\\
        &\leq
        (q-1)
        \sum_{|\beta|<q}
        |c_\beta|^2
        \|H_\beta\|_{L^2(\gamma^{n+1})}^2\\
        &= 
        (q-1)
        \int_{\bbR^{n+1}}p^2\,\ud\gamma^{n+1}.
\end{align}
Consequently, for every nonzero polynomial $p\in \calP_{<q}$ the Rayleigh quotient satisfies
\begin{align}\label{eq:gaussian-split-lower-polynomial-Gaussian}
        \frac{
        \int_{\bbR^{n+1}}|\nabla p|^2\,\ud\gamma^{n+1}
        }{
        \int_{\bbR^{n+1}}p^2\,\ud\gamma^{n+1}
        }
        \leq
        q-1.
\end{align}
Since the multiplicity of the Gaussian eigenvalue $j$ is
$\binom{n+j}{j}$, exactly $\sum_{j=0}^{q-1}\binom{n+j}{j}=N_q$ eigenvalues, including the zero
eigenvalue and counting multiplicity, lie strictly below $q$. This proves
\eqref{eq:gaussian-split-Gaussian-index}.

We next estimate the Rayleigh quotient of the same polynomials on the cone
with potential $r^2/2$. For every $0\neq p\in\cP_{<q}$, we have
\begin{align}\label{eq:gaussian-split-lower-polynomial-cone-one}
        \cR_{g_C,r^2/2}(p)
        \leq
        A(\bar g,s)
        \frac{
        \int_{\bbR^{n+1}}|\nabla p|^2\,\ud\gamma^{n+1}
        }{
        \int_{\bbR^{n+1}}p^2\,\ud\gamma^{n+1}
        }
        \leq
        A(\bar g,s)(q-1)
        <
        q,
\end{align}
where for the first inequality in the above display, we used
\begin{align}
        e^{-r^2/2}\dvol_{g_C}
        =
        s^ne^{-r^2/2}r^n\,\ud r\,\dvol_{\bar g}
        =
        s^ne^{-r^2/2}r^n\,\ud r\,\dvol_{g_0}
        =
        s^n(2\pi)^{\frac{n+1}{2}}\,\ud\gamma^{n+1},
\end{align}
and
\begin{align}
        |\nabla p|_{g_C}^2
        =
        |\partial_rp|^2
        +
        \frac1{s^2r^2}|\nabla_{\bbS^n}p|_{\bar g}^2
        \leq
        A(\bar g,s)
        \left(
        |\partial_rp|^2
        +
        \frac1{r^2}|\nabla_{\bbS^n}p|_{g_0}^2
        \right)
        =
        A(\bar g,s)|\nabla p|^2,
\end{align}
where the final norm is the Euclidean norm written in polar coordinates. The second inequality in \eqref{eq:gaussian-split-lower-polynomial-cone-one} follows from \eqref{eq:gaussian-split-lower-polynomial-Gaussian}, while the last inequality follows from \eqref{eq:gaussian-split-polynomial-assumptions}. 

Now observe that after rescaling the polynomials in $\cP_{<q}$ we obtain
\begin{align}\label{eq:gaussian-split-polynomial-alpha-scaling}
        \sup_{0\neq p\in\cP_{<q}}
        \cR_{g_C,\alpha r^2/2}(p)
        =
        \alpha
        \sup_{0\neq p\in\cP_{<q}}
        \cR_{g_C,r^2/2}(p).
\end{align}
Therefore, the estimate \eqref{eq:gaussian-split-lower-polynomial-cone-one} implies
\begin{align}\label{eq:gaussian-split-lower-polynomial-cone-alpha}
        \sup_{0\neq p\in\cP_{<q}}
        \cR_{g_C,\alpha r^2/2}(p)
        \leq
        \alpha A(\bar g,s)(q-1),
\end{align}
for any $\alp>0.$ Now we aim to construct the vector space of polynomials $E_C$ as stated in the lemma. This space will consist of $\cP_{<q}$ and the following polynomial written in polar coordinates
\begin{align}\label{eq:gaussian-split-Pq-polar}
        P_q(r,\omega)=r^qf_q(\omega),
\end{align}
where $f_q(z,y):=\operatorname{Re}(z^q)$ for $(z,y)\in \bbS^{n}\subset \bbC\times \bbR^{n-1}$. We will now estimate the Rayleigh quotient of $P_q$. To this end, define for $m>-1$
\begin{align}
        I_m(\alpha)
        :=
        \int_0^\infty r^me^{-\alpha r^2/2}\,\ud r.
\end{align}
Integration by parts implies
\begin{align}\label{eq:gaussian-split-Im-recurrence}
        I_m(\alpha)
        =
        \frac{m-1}{\alpha}I_{m-2}(\alpha)
        \qquad
        \text{for every }m>1.
\end{align}
Since
\begin{align}
        |\nabla P_q|_{g_C}^2
        =
        r^{2q-2}
        \left(
        q^2f_q^2+s^{-2}|\nabla f_q|_{\bar g}^2
        \right),
\end{align}
we obtain
\begin{align}\label{eq:gaussian-split-Pq-Rayleigh}
        \cR_{g_C,\alpha r^2/2}(P_q)
        &=
        \frac{
        \left(
        q^2\int_{\bbS^n}f_q^2\dvol_{\bar g}
        +s^{-2}\int_{\bbS^n}|\nabla f_q|_{\bar g}^2\dvol_{\bar g}
        \right)
        I_{n+2q-2}(\alpha)
        }{
        \left(\int_{\bbS^n}f_q^2\dvol_{\bar g}\right)
        I_{n+2q}(\alpha)
        }
        \notag\\
        &=
        \frac{\alpha}{n+2q-1}
        \left(
        q^2+s^{-2}\cR_{\bar g}(f_q)
        \right),
\end{align}
where in the last equality  we used
$I_{n+2q}(\alpha)=\frac{n+2q-1}{\alpha}I_{n+2q-2}(\alpha)$.
By \eqref{eq:gaussian-split-polynomial-assumptions},
\begin{align}
        q^2+s^{-2}\cR_{\bar g}(f_q)
        <
        q^2+q(q+n-1)
        =
        q(n+2q-1).
\end{align}
Therefore, if
\begin{align}\label{eq:gaussian-split-Bq}
        B_q
        :=
        \frac{q^2+s^{-2}\cR_{\bar g}(f_q)}{n+2q-1},
\end{align}
then
\begin{align}
        0<B_q<q,
        \qquad
        \cR_{g_C,\alpha r^2/2}(P_q)=\alpha B_q.
\end{align}

Since $q\geq2$, the quantities
$q/[A(\bar g,s)(q-1)]$ and $q/B_q$ are well-defined and strictly larger than
one. We may therefore choose
\begin{align}\label{eq:gaussian-split-alpha-choice}
        1<\alpha<
        \min\left\{
        \frac{q}{A(\bar g,s)(q-1)},
        \frac{q}{B_q}
        \right\}.
\end{align}
For this choice of $\alpha$, we have
\begin{align}\label{eq:gaussian-split-two-polynomial-bounds}
        \sup_{0\neq p\in\cP_{<q}}
        \cR_{g_C,V_C}(p)
        &\leq
        \alpha A(\bar g,s)(q-1)
        <q,\notag\\
        \cR_{g_C,V_C}(P_q)
        &=
        \alpha B_q
        <q.
\end{align}
It remains to show that $P_q$ is orthogonal to $\cP_{<q}$ with respect to the Gaussian-weighted $L^2$ inner product and Dirichlet form on the cone. For $u,v \in W^{1,2} \left(C,g_C,e^{-V_C}\dvol_{g_C}\right)$, define
\begin{align}
        \cB_0(u,v)
        :=
        \int_Cuv\,e^{-V_C}\dvol_{g_C},
        \quad 
        \cB_1(u,v)
        :=
        \int_C
        \la\nabla u,\nabla v\ra_{g_C}
        e^{-V_C}\dvol_{g_C}.
\end{align} 
Then, observe that the transformation $R_\vartheta(r,z,y):=(r,e^{\iii\vartheta}z,y)$ preserves the metric $g_C$ and the potential $V_C$. Then arguing as in \eqref{eqn:change-of-variables} earlier, we obtain that
\begin{align}\label{eq:gaussian-split-polynomial-orthogonality}
        \int_CP_qp\,e^{-V_C}\dvol_{g_C}
        =0,\quad 
        \int_C\la\nabla P_q,\nabla p\ra_{g_C}
        e^{-V_C}\dvol_{g_C}
        =0,
\end{align}
for every $p\in \cP_{<q}.$ Set $E_C :=\cP_{<q}\oplus\spn\{P_q\}.$ By \eqref{eq:gaussian-split-polynomial-count}, $\dim E_C=N_q+1$, which proves
\eqref{eq:gaussian-split-polynomial-dimension}. Finally, if
$0\neq u=p+aP_q\in E_C$, where $p\in\cP_{<q}$ and $a\in\bbR$, then
\begin{align}
        \cR_{g_C,V_C}(u)
        &=
        \frac{
        \int_C|\nabla p|_{g_C}^2e^{-V_C}\dvol_{g_C}
        +a^2\int_C|\nabla P_q|_{g_C}^2e^{-V_C}\dvol_{g_C}
        }{
        \int_Cp^2e^{-V_C}\dvol_{g_C}
        +a^2\int_CP_q^2e^{-V_C}\dvol_{g_C}
        }
        \notag\\
        &\leq
        \max\left\{
        \sup_{0\neq p\in\cP_{<q}}\cR_{g_C,V_C}(p),
        \cR_{g_C,V_C}(P_q)
        \right\}
        <q.
\end{align}
Here the first equality follows from
\eqref{eq:gaussian-split-polynomial-orthogonality}, and the last inequality
follows from \eqref{eq:gaussian-split-two-polynomial-bounds}. This proves
\eqref{eq:gaussian-split-polynomial-Rayleigh} and completes the proof.
\end{proof}

Unfortunately, the cone constructed above has a singularity at the origin. We now prove a lemma that will smooth this singularity without losing the curvature control on the metric and the potential and the spectral inequalities needed to prove Theorem~\ref{thm:gaussian-split-main}. We first recall a technical proposition, based on
\cite[Proposition~4.1]{barGauduchonMoroianu2005}, which will be useful for
computing the Ricci curvature of the smoothed metric.

\begin{prop}\label{prop:generalized-cylinder-curvature}
Let $M^n$ be a smooth manifold, let $I\subset\bbR$ be an interval, and let
$\{h_r\}_{r\in I}$ be a smooth family of Riemannian metrics on $M$. Set
\begin{align}
        Z
        :=
        I\times M,
        \qquad
        g^Z
        :=
        \ud r^2+h_r.
\end{align}
For $r\in I$, let $M_r:=\{r\}\times M$, let $\nu:=\partial_r$, and let
$W_r$ denote the Weingarten map of $M_r$ with respect to $\nu$, so that
\begin{align}\label{eq:generalized-cylinder-Weingarten}
        \la W_r(X),Y\ra_{h_r}
        =
        -\frac12\partial_rh_r(X,Y)
\end{align}
for $X,Y$ tangent to $M_r$. Then
\begin{align}\label{eq:generalized-cylinder-Ricci-radial}
        \Ric_{g^Z}(\partial_r,\partial_r)
        =
        \partial_r\Tr(W_r)-\Tr(W_r^2),
\end{align}
\begin{align}\label{eq:generalized-cylinder-Ricci-mixed}
        \Ric_{g^Z}(\partial_r,X)
        =
        X\bigl(\Tr(W_r)\bigr)
        -
        \la\div_{h_r}W_r,X\ra_{h_r},
\end{align}
and
\begin{align}\label{eq:generalized-cylinder-Ricci-tangential}
        \Ric_{g^Z}(X,Y)
        =
        \Ric_{h_r}(X,Y)
        +
        \la
        \left(
        \partial_rW_r-\Tr(W_r)W_r
        \right)X,
        Y
        \ra_{h_r}.
\end{align}
Here $\partial_rW_r$ is computed after identifying the tangent bundles of
the slices through the product structure.
\end{prop}

\begin{proof}
These identities follow from Proposition~4.1 in \cite{barGauduchonMoroianu2005}. To see this, first using
\eqref{eq:generalized-cylinder-Weingarten} we obtain
\begin{align}
        \partial_r\Tr(W_r)
        =
        2\Tr(W_r^2)
        -
        \frac12\Tr_{h_r}(\partial_r^2h_r),
\end{align}
which transforms equation~(18) in Proposition~4.1 in \cite{barGauduchonMoroianu2005} into
\eqref{eq:generalized-cylinder-Ricci-radial}. Moreover, differentiating \eqref{eq:generalized-cylinder-Weingarten} gives
\begin{align}
        \frac12\partial_r^2h_r(X,Y)
        =
        2\la W_r(X),W_r(Y)\ra_{h_r}
        -
        \la(\partial_rW_r)(X),Y\ra_{h_r}.
\end{align}
Substituting this identity into equation~(20) in Proposition~4.1 in \cite{barGauduchonMoroianu2005} gives \eqref{eq:generalized-cylinder-Ricci-tangential}, while equation~(19) in Proposition~4.1 in \cite{barGauduchonMoroianu2005} gives \eqref{eq:generalized-cylinder-Ricci-mixed}.
\end{proof}

\begin{lem}\label{lem:gaussian-split-smoothing}
Let $n,q,s,\rho,\bar g$, and $\{g_\ta\}_{\ta\in[0,1]}$ be the metrics
constructed in Lemma~\ref{lem:gaussian-split-link}, and let
$\alpha$, $(C,g_C,V_C)$, $E_C$, and $N_q$ be as constructed in
Lemma~\ref{lem:gaussian-split-polynomial-cone}. Then, for every sufficiently
small $\eta>0$, there exist a smooth complete metric $g_\eta$ on
$\bbR^{n+1}$, a smooth function $V_\eta:\bbR^{n+1}\to\bbR$, and a
finite-dimensional vector space
\begin{align}
        E_\eta
        \subset
        C^\infty(\bbR^{n+1})
        \cap
        W^{1,2}
        \left(
        \bbR^{n+1},g_\eta,e^{-V_\eta}\dvol_{g_\eta}
        \right)
\end{align}
such that
\begin{align}\label{eq:gaussian-split-smoothed-curvature}
        \Ric_{g_\eta}\geq0,
        \quad
        \nabla_{g_\eta}^2V_\eta\geq g_\eta,
\end{align}
$\dim E_\eta=N_q+1,$ and
\begin{align}\label{eq:gaussian-split-smoothed-Rayleigh}
        \sup_{0\neq f\in E_\eta}
        \cR_{g_\eta,V_\eta}(f)
        <
        q.
\end{align}
Moreover,
\begin{align}\label{eq:gaussian-split-smoothed-finite-mass}
        \int_{\bbR^{n+1}}e^{-V_\eta}\dvol_{g_\eta}<\infty.
\end{align}
% and the embedding
% \begin{align}\label{eq:gaussian-split-smoothed-compact-embedding}
%         W^{1,2}
%         \left(
%         \bbR^{n+1},g_\eta,e^{-V_\eta}\dvol_{g_\eta}
%         \right)
%         \hookrightarrow
%         L^2
%         \left(
%         \bbR^{n+1},e^{-V_\eta}\dvol_{g_\eta}
%         \right)
% \end{align}
% is compact. In particular, the corresponding weighted Laplacian has compact
% resolvent.
\end{lem}

\begin{proof}
For notational convenience, set $E:=E_C$,  $\Lambda_E:=\sup_{0\neq p\in E_C} \cR_{g_C,V_C}(p)$ and $\delta_*:=\rho-s>0$. By \eqref{eq:gaussian-split-polynomial-Rayleigh}, we have $\Lambda_E<q.$ Set $P_\ta := g_\ta^{-1}\partial_\ta g_\ta.$ By Lemma~\ref{lem:gaussian-split-link}, for every $\ta\in[0,1]$ we have
\begin{align}\label{eq:gaussian-split-smoothing-Ricci-hypothesis}
        \Ric_{g_\ta}
        \geq
        (n-1)\rho^2g_\ta,\text{ and } \Tr(P_\ta)=0.
\end{align}
Our argument will proceed in five steps.
\paragraph{$\mathrm{(i)}$ \textit{Construction of a smooth metric}} Choose smooth functions $\psi:\bbR\to[s,1]$ and $\chi:\bbR\to[0,1]$ such that
\begin{align}
        &\psi(t)=1
        \text{ for }t\leq0,
        \quad
        \psi(t)=s
        \text{ for }t\geq4,
        \quad
        \psi'\leq0,\quad 
        \psi(t)
        \leq
        s+\frac{\delta_*}{4}
        \text{ for }t\geq2, \qquad \label{eq:gaussian-split-profile-small}\\
        &\chi(t)=0 \text{ for }t\leq2,\quad 
        \chi(t)=1 \text{ for }t\geq3,\quad 
        \chi'(t)>0 \text{ for }2<t<3,\quad 
        -\psi'(t)
        \geq
        c_0\delta_*
        \text{ for }2\leq t\leq 3 \qquad \label{eq:gaussian-split-profile-negative}
\end{align}
for some constant $c_0>0$.  For $T\geq1$, define $\psi_T(t):=\psi(t/T)$ and $\chi_T(t):=\chi(t/T).$ For $r>0$, set
\begin{align}\label{eq:gaussian-split-phi-u-definition}
        \phi(r)
        :=
        \int_0^r \psi_T(\log\sigma)\,\ud\sigma,
        \qquad
        u(r):=\chi_T(\log r).
\end{align}
Since $\psi_T(\log\sigma)=1$ and $\chi_T(\log\sigma)=0$ for all sufficiently small $\sigma>0$, we have $\phi(r)=r$ and $u(r)=0$ near $r=0$. Moreover,
\begin{align}\label{eq:gaussian-split-phi-bounds}
        s\leq\phi'\leq1,
        \qquad
        \phi''\leq0,
        \qquad
        sr\leq\phi(r)\leq r.
\end{align}
There exists $b\geq0$ such that, for every $r\geq e^{4T}$,
\begin{align}\label{eq:gaussian-split-phi-end}
        u(r)=1,
        \qquad
        \phi(r)=s(r+b).
\end{align}
Define $ \widetilde g :=\ud r^2+\phi(r)^2g_{u(r)}.$ Since $\phi(r)=r$ and $g_{u(r)}=g_0=g_\can$ near the origin, the metric
$\widetilde g$ extends smoothly across the origin. Moreover,
$\widetilde g\geq\ud r^2$, and hence $\widetilde g$ is complete.

\paragraph{$\mathrm{(ii)}$ \textit{Ricci curvature lower bound}} Set
$h_r:=\phi(r)^2g_{u(r)}$, so that
$\widetilde g=\ud r^2+h_r$, and set
\begin{align}
        F_r
        :=
        \frac12u'(r)P_{u(r)}.
\end{align}
Since $\partial_\ta g_\ta$ is symmetric, $P_\ta$ is self-adjoint with
respect to $g_\ta$. Consequently, $F_r$ is self-adjoint with respect to
$h_r$. By \eqref{eq:generalized-cylinder-Weingarten}, the Weingarten map of the slice $M_r=\{r\}\times\bbS^n$ satisfies
\begin{align}\label{eq:gaussian-split-shape-operator}
        -W_r
        =
        \frac{\phi'}{\phi}\Id+F_r.
\end{align}
In particular, if
$\mathrm{II}(X,Y):=-h_r(W_r(X),Y)$, then
\begin{align}
        \mathrm{II}(X,Y)
        =
        h_r
        \left(
        \left(
        \frac{\phi'}{\phi}\Id+F_r
        \right)X,Y
        \right).
\end{align}
Since $\Tr(P_\ta)=0$, we have
$\Tr(F_r)=0$ and
$\Tr(W_r)=-n\phi'/\phi$. Therefore,
Proposition~\ref{prop:generalized-cylinder-curvature} gives
\begin{align}\label{eq:gaussian-split-Ricci-rr-exact}
        \Ric_{\widetilde g}(\partial_r,\partial_r)
        &=
        -n\frac{\phi''}{\phi}
        -
        \Tr(F_r^2),\notag\\
        \Ric_{\widetilde g}(\partial_r,X)
        &=
        \la\div_{h_r}F_r,X\ra_{h_r},\notag\\
        \Ric_{\widetilde g}(X,Y)
        &=
        \Ric_{g_{u(r)}}(X,Y)
        -
        \left(
        \frac{\phi''}{\phi}
        +(n-1)\frac{(\phi')^2}{\phi^2}
        \right)h_r(X,Y)
        -
        h_r
        \left(
        \left(
        \partial_rF_r
        +
        n\frac{\phi'}{\phi}F_r
        \right)X,Y
        \right)
\end{align}
for vectors $X,Y$ tangent to $M_r$. Recall that $\delta_*:=\rho-s>0$ and set
\begin{align}
        M
        :=
        1+
        \sup_{\ta\in[0,1]}
        \left(
        |P_\ta|_{C^1(g_\ta)}
        +
        |\partial_\ta P_\ta|_{C^0(g_\ta)}
        \right).
\end{align}
Let $\cA_T:=\supp u'$. In the remainder of this step, $c>0$ and
$C<\infty$ denote constants independent of $T$. Since
$u(r)=\chi((\log r)/T)$, we have
\begin{align}
        |u'(r)|
        &\leq
        \frac{C}{Tr},
        \qquad
        |u''(r)|
        \leq
        \frac{C}{Tr^2}.
\end{align}
Moreover,
\begin{align}
        \partial_rF_r
        &=
        \frac12u''(r)P_{u(r)}
        +
        \frac12(u'(r))^2
        \left.
        \partial_\ta P_\ta
        \right|_{\ta=u(r)},
        \notag\\
        \div_{h_r}F_r
        &=
        \frac12u'(r)
        \div_{h_r}P_{u(r)}.
\end{align}
Using these identities, $\phi\leq r$, and the fact that
$h_r=\phi^2g_{u(r)}$ is a constant rescaling on each slice, we obtain
\begin{align}\label{eq:gaussian-split-F-bound}
        |F_r|_{h_r}
        \leq
        \frac{CM}{T\phi},\quad 
        |\div_{h_r}F_r|_{h_r}
        \leq
        \frac{CM}{T\phi^2},\quad 
        \left|
        \partial_rF_r + n\frac{\phi'}{\phi}F_r
        \right|_{h_r}
        \leq
        \frac{CM}{T\phi^2}.
\end{align}
On $\cA_T$, we have $(\log r)/T\in[2,3]$. Therefore,
\eqref{eq:gaussian-split-profile-negative},
\eqref{eq:gaussian-split-profile-small}, and
$sr\leq\phi(r)$ give
\begin{align}
        -\frac{\phi''}{\phi}
        \geq
        \frac{c\delta_*}{T\phi^2},
        \quad
        \rho^2-(\phi')^2
        \geq
        c\delta_*.
\end{align}
It follows from \eqref{eq:gaussian-split-Ricci-rr-exact} and
\eqref{eq:gaussian-split-F-bound} that
\begin{align}
        \Ric_{\widetilde g}(\partial_r,\partial_r)
        \geq
        \frac{c\delta_*}{T\phi^2}
        -
        \frac{CM^2}{T^2\phi^2},\quad 
        \left|
        \Ric_{\widetilde g}(\partial_r,X)
        \right|
        \leq
        \frac{CM}{T\phi^2}|X|_{h_r},\quad
        \Ric_{\widetilde g}(X,X)
        \geq
        \left(
        \frac{c\delta_*}{\phi^2}
        -
        \frac{CM}{T\phi^2}
        \right)|X|_{h_r}^2.
\end{align}
Let $Y=a\partial_r+X$, where $X$ is tangent to $M_r$. Young's inequality
gives
\begin{align}
        \Ric_{\widetilde g}(Y,Y)
        &\geq
        \frac{c\delta_*}{4T\phi^2}a^2
        +
        \frac{c\delta_*}{4\phi^2}|X|_{h_r}^2
        \geq
        0
\end{align}
provided that
\begin{align}\label{eq:gaussian-split-T-choice}
        T
        \geq
        \frac{CM^2}{\delta_*^2}.
\end{align}
Outside $\cA_T$, we have $F_r=0$ and $u(r)\in\{0,1\}$. Hence
\eqref{eq:gaussian-split-Ricci-rr-exact} reduces to
\begin{align}
        \Ric_{\widetilde g}(\partial_r,\partial_r)
        &=
        -n\frac{\phi''}{\phi}
        \geq
        0,
        \quad
        \Ric_{\widetilde g}(\partial_r,X)
        =
        0,
        \notag\\
        \Ric_{\widetilde g}(X,X)
        &=
        \Ric_{g_{u(r)}}(X,X)
        -
        \left(
        (n-1)(\phi')^2+\phi\phi''
        \right)g_{u(r)}(X,X).
\end{align}
If $u(r)=0$, then
$\Ric_{g_0}=(n-1)g_0$ and $\phi'\leq1$. If $u(r)=1$, then
$\Ric_{\bar g}\geq(n-1)\rho^2\bar g$ and
$\phi'\leq s+\delta_*/4<\rho$. Since $\phi''\leq0$, the tangential
component is nonnegative in either case. We conclude that $\Ric_{\widetilde g}\geq 0.$
\paragraph{$\mathrm{(iii)}$ \textit{Construction of the potential}} Define $\widetilde V$ by
\begin{align}\label{eq:gaussian-split-Vtilde-definition}
        \widetilde V'(r)
        :=
        \alpha\frac{\phi(r)}{\phi'(r)},
        \qquad
        \widetilde V(0):=0.
\end{align}
Since $\phi'\geq s>0$, the function $\widetilde V$ is smooth. Near the
origin, $\phi(r)=r$, and hence
$\widetilde V(r)=\alpha r^2/2$. In the radial direction,
\begin{align}\label{eq:gaussian-split-Vtilde-radial-Hessian}
        \nabla_{\widetilde g}^2\widetilde V(\partial_r,\partial_r)
        =
        \widetilde V''
        =
        \alpha
        \left(
        1-
        \frac{\phi\phi''}{(\phi')^2}
        \right)
        \geq
        \alpha
        >
        1.
\end{align}
For a vector $X$ tangent to a slice,
\begin{align}
        \nabla_{\widetilde g}^2\widetilde V(X,X)
        =
        \widetilde V'\mathrm{II}(X,X).
\end{align}
Using \eqref{eq:gaussian-split-shape-operator},
\eqref{eq:gaussian-split-F-bound}, and $\phi'\geq s$, we obtain
\begin{align}
        \mathrm{II}(X,X)
        \geq
        \left(
        \frac{\phi'}{\phi}
        -
        \frac{CM}{T\phi}
        \right)|X|_{h_r}^2
        \geq
        \left(
        1-
        \frac{CM}{sT}
        \right)
        \frac{\phi'}{\phi}|X|_{h_r}^2.
\end{align}
Therefore,
\begin{align}\label{eq:gaussian-split-Vtilde-tangential-Hessian}
        \nabla_{\widetilde g}^2\widetilde V(X,X)
        \geq
        \alpha
        \left(
        1-
        \frac{CM}{sT}
        \right)|X|_{h_r}^2.
\end{align}
Since $\alpha>1$, we may increase $T$, while preserving
\eqref{eq:gaussian-split-T-choice}, so that
\begin{align}
        \alpha
        \left(
        1-
        \frac{CM}{sT}
        \right)
        \geq
        1.
\end{align}
Since $\widetilde V=\widetilde V(r)$, we have
\begin{align}
        \nabla_{\widetilde g}\widetilde V
        =
        \widetilde V'\partial_r.
\end{align}
Moreover, since $\widetilde g=\ud r^2+h_r$, the curves
$r\mapsto(r,\omega)$ are unit-speed geodesics, and hence
$\nabla_{\partial_r}^{\widetilde g}\partial_r=0$. Therefore, for every
vector $X$ tangent to $M_r$,
\begin{align}
        \nabla_{\widetilde g}^2\widetilde V(\partial_r,X)
        =
        \la
        \nabla_{\partial_r}^{\widetilde g}
        (
        \widetilde V'\partial_r
        ),
        X
        \ra_{\widetilde g}=
        \widetilde V''
        \la\partial_r,X\ra_{\widetilde g}
        +
        \widetilde V'
        \la
        \nabla_{\partial_r}^{\widetilde g}\partial_r,
        X
        \ra_{\widetilde g}=
        0.
\end{align}
Now let $Y=a\partial_r+X$, where $X$ is tangent to $M_r$. Since
$\widetilde g(\partial_r,X)=0$, it follows from
\eqref{eq:gaussian-split-Vtilde-radial-Hessian} and
\eqref{eq:gaussian-split-Vtilde-tangential-Hessian} that
\begin{align}
        \nabla_{\widetilde g}^2\widetilde V(Y,Y)
        =
        a^2
        \nabla_{\widetilde g}^2\widetilde V
        (\partial_r,\partial_r)
        +
        2a
        \nabla_{\widetilde g}^2\widetilde V
        (\partial_r,X)
        +
        \nabla_{\widetilde g}^2\widetilde V(X,X)\geq
        \alpha a^2+|X|_{h_r}^2\geq
        a^2+|X|_{h_r}^2=
        |Y|_{\widetilde g}^2.
\end{align}
Consequently, $\nabla_{\widetilde g}^2\widetilde V \geq \widetilde g.$ By \eqref{eq:gaussian-split-phi-end}, after setting $R=r+b$, we have, outside
a compact set,
\begin{align}\label{eq:gaussian-split-gtilde-end}
        \widetilde g
        =
        \ud R^2+s^2R^2\bar g,
        \qquad
        \widetilde V
        =
        \frac{\alpha R^2}{2}+c_*
\end{align}
for some constant $c_*\in\bbR$.

\paragraph{$\mathrm{(iv)}$ \textit{Shrinking the cap and transferring the test space}} Let
$\Phi_\eta:\bbR^{n+1}\to\bbR^{n+1}$ be a radial dilation map $\Phi_\eta(r,\omega)=(r/\eta,\omega)$. Next, define the rescaled metric $g_\eta :=\eta^2\Phi_\eta^*\widetilde g$ and potential $V_\eta:=\eta^2\Phi_\eta^*\widetilde V.$
The metric $g_\eta$ is smooth and complete. Moreover,
\begin{align}
        \Ric_{g_\eta}
        \geq
        0,\quad 
        \nabla_{g_\eta}^2V_\eta
        =
        \eta^2\Phi_\eta^*
        (
        \nabla_{\widetilde g}^2\widetilde V
        )
        \geq
        \eta^2\Phi_\eta^*\widetilde g
        =
        g_\eta.
\end{align}
This proves \eqref{eq:gaussian-split-smoothed-curvature}. Choose $r_*=e^{4T}$, then $u(r)=1$ and $\phi(r)=s(r+b)$ for every $r\geq r_*$. Then when  $r>\eta r_*$,
\begin{align}
        g_\eta
        =
        \ud r^2+s^2(r+b\eta)^2\bar g,\quad 
        V_\eta
        =
        \frac{\alpha(r+b\eta)^2}{2}+c_\eta,
\end{align}
where $c_\eta:=\eta^2c_*$. Set
\begin{align}
        R_0
        :=
        r_*+b,
        \quad
        K_\eta
        :=
        \overline{B_{\eta r_*}(0)},
\end{align}
where the ball is taken with respect to the Euclidean metric. Define
\begin{align}
        \Psi_\eta:
        (R_0\eta,\infty)\times\bbS^n
        &\rightarrow
        \bbR^{n+1}\setminus K_\eta,
        \\
        \Psi_\eta(R,\omega)
        &:=
        (R-b\eta)\omega.
\end{align}
Then $\Psi_\eta$ is a diffeomorphism. It extends smoothly to
$[R_0\eta,\infty)\times\bbS^n$, and its extension maps
$\{R_0\eta\}\times\bbS^n$ onto $\partial K_\eta$. Finally,
\begin{align}\label{eq:gaussian-split-geta-end}
        \Psi_\eta^*g_\eta
        =
        \ud R^2+s^2R^2\bar g,\quad 
        \Psi_\eta^*V_\eta
        =
        \frac{\alpha R^2}{2}+c_\eta.
\end{align}
We now define the rescaled space of polynomials $E_\eta$. Choose a smooth function $\zeta_\eta:[0,\infty)\to[0,1]$ such that
$\zeta_\eta(R)=0$ for $R\leq\frac32R_0\eta$,
$\zeta_\eta(R)=1$ for $R\geq2R_0\eta$, and
$|\zeta_\eta'(R)|\leq C\eta^{-1}$ for every $R\geq0$. For $p\in E$ and
$(R,\omega)\in(0,\infty)\times\bbS^n$, set
\begin{align}
        p(R,\omega)
        :=
        p(R\omega),
\end{align}
where on the right-hand side $p$ is regarded as an ordinary polynomial
on $\bbR^{n+1}$. Define $T_\eta:E \rightarrow C^\infty(\bbR^{n+1})$ by
\begin{align}\label{eq:gaussian-split-Teta-definition}
        (T_\eta p)(x)
        :=
        \begin{cases}
        \zeta_\eta(R)p(R,\omega),
        &
        \text{if }
        x=\Psi_\eta(R,\omega)
        \in\bbR^{n+1}\setminus K_\eta,\\
        0,
        &
        \text{if }
        x\in K_\eta.
        \end{cases}
\end{align}
Note that since $\zeta_\eta(R)=0$ whenever $R_0\eta
        <
        R
        \leq
        \frac32R_0\eta$, $T_\eta p$ vanishes on an open collar of $\partial K_\eta$ and thus the transition between the two regions in \eqref{eq:gaussian-split-Teta-definition} is smooth. Moreover, $T_\eta p$ has polynomial growth and therefore, $T_\eta p \in W^{1,2}(\bbR^{n+1},g_\eta,e^{-V_\eta}\dvol_{g_\eta}).$ Define, $E_\eta:=T_\eta(E)$. If $T_\eta p=0$, then $p$ vanishes on the
nonempty open set $\{R>2R_0\eta\}$. Since $p$ is an ordinary polynomial,
it follows that $p\equiv0$. Thus $T_\eta:E_C\to E_\eta$ is injective, and
hence $\dim E_\eta = \dim E_C = N_q+1$. 

Let $p\in E_C$. Writing $p$ as a finite sum of homogeneous polynomials implies
\begin{align}\label{eq:gaussian-split-polynomial-tip-bound}
        |p|+|\nabla p|_{g_C}
        \leq
        C_p
        \qquad
        \text{for }0<R\leq1,
\end{align}
where $C_p>0$ is a constant depending on the polynomial $p.$ Using $\dvol_{g_C}=s^nR^n\,\ud R\,\dvol_{\bar g}$ and
$|\zeta_\eta'|\leq C\eta^{-1}$, we obtain
\begin{align}
        \int_{\{R\leq2R_0\eta\}}
        (p^2
        +
        |\nabla p|_{g_C}^2)e^{-V_C}\dvol_{g_C}
        =
        O(\eta^{n+1}),\quad 
        \int_{\{R_0\eta\leq R\leq2R_0\eta\}}
        |\zeta_\eta'|^2p^2e^{-V_C}\dvol_{g_C}
        =
        O(\eta^{n-1}).
\end{align}
The cross term arising from $\nabla(\zeta_\eta p)=\zeta_\eta\nabla p+p\nabla\zeta_\eta$ tends to zero by the Cauchy--Schwarz inequality. Since $n\geq4$, both displayed error terms also tend to zero as $\eta\to 0$. Choose a basis $p_1,\ldots,p_m$ of $E_C$, where $m=\dim E_C$. Define the matrices
\begin{align}
        G_\eta^{ij}
        :=
        e^{c_\eta}
        \int_{\bbR^{n+1}}
        T_\eta p_iT_\eta p_j
        e^{-V_\eta}\dvol_{g_\eta},\quad 
        D_\eta^{ij}
        :=
        e^{c_\eta}
        \int_{\bbR^{n+1}}
        \la\nabla T_\eta p_i,\nabla T_\eta p_j\ra_{g_\eta}
        e^{-V_\eta}\dvol_{g_\eta}.
\end{align}
The preceding estimates and \eqref{eq:gaussian-split-geta-end} imply that as $\eta \to 0$ we have
\begin{align}\label{eq:gaussian-split-L2-matrix-convergence}
        G_\eta^{ij}
        \rightarrow
        G_0^{ij}
        :=
        \int_Cp_ip_j e^{-V_C}\dvol_{g_C},\quad 
        D_\eta^{ij}
        \rightarrow
        D_0^{ij}
        :=
        \int_C
        \la\nabla p_i,\nabla p_j\ra_{g_C}
        e^{-V_C}\dvol_{g_C}.
\end{align}
For $a=(a_1,\ldots,a_m)\in\bbR^m$, set
$p_a:=\sum_{i=1}^ma_ip_i$ and
$p_{a,\eta}:=\sum_{i=1}^ma_iT_\eta p_i$. Since $G_0$ is positive definite,
there exists $c_E>0$ such that
\begin{align}
        \sum_{i,j=1}^mG_0^{ij}a_ia_j
        \geq
        c_E|a|^2
        \qquad
        \text{for every }a\in\bbR^m.
\end{align}
Since there are finitely many matrix entries, the corresponding quadratic
forms converge uniformly on $\bbS^{m-1}$. Moreover, since $G_0$ is
positive definite, $G_\eta$ is positive definite for every sufficiently
small $\eta>0$. Consequently, as $\eta \to 0$ we obtain
\begin{align}
        \sup_{0\neq f\in E_\eta}
        \cR_{g_\eta,V_\eta}(f)
        &=
        \sup_{a\in \bbS^{m-1}}
        \frac{
        \sum_{i,j=1}^mD_\eta^{ij}a_ia_j
        }{
        \sum_{i,j=1}^mG_\eta^{ij}a_ia_j
        }
        \rightarrow
        \sup_{a\in \bbS^{m-1}}
        \frac{
        \sum_{i,j=1}^mD_0^{ij}a_ia_j
        }{
        \sum_{i,j=1}^mG_0^{ij}a_ia_j
        }
        =
        \Lambda_E.
\end{align}
Since $\Lambda_E<q$, the above convergence implies that
\begin{align}
        \sup_{0\neq f\in E_\eta}
        \cR_{g_\eta,V_\eta}(f)
        <
        q
\end{align}
for every sufficiently small $\eta>0$. This proves
\eqref{eq:gaussian-split-smoothed-Rayleigh}.

\paragraph{$\mathrm{(v)}$ \textit{Integrability of the potential}} It remains to prove \eqref{eq:gaussian-split-smoothed-finite-mass}. On the exact conical end, \eqref{eq:gaussian-split-geta-end} gives
\[
        e^{-V_\eta}\dvol_{g_\eta}
        =
        e^{-c_\eta}s^n
        e^{-\alpha R^2/2}R^n
        \,\ud R\,\dvol_{\bar g}.
\]
The right-hand side is integrable. The complement of the conical end is compact, and $V_\eta$ and $g_\eta$ are smooth there. Consequently,
\begin{align}
        \int_{\bbR^{n+1}}
        e^{-V_\eta}\dvol_{g_\eta}
        <
        \infty.
\end{align}
This proves \eqref{eq:gaussian-split-smoothed-finite-mass}.
\end{proof}

\begin{proof}[Proof of Theorem~\ref{thm:gaussian-split-main}]
Let $d\geq5$ and set $n:=d-1\geq4$. From Lemma~\ref{lem:gaussian-split-link} we obtain an integer $q\geq2$, which is the degree of the test function $f_q(z,y)=\Re(z^q)$, a cone scaling
parameter $s$, a Ricci curvature lower-bound parameter $\rho$ satisfying $0<s<\rho<1$, a smooth metric $\bar g$ on $\bbS^n$, and
a smooth family of metrics $\{g_\ta\}_{\ta\in[0,1]}$ joining $g_0=g_\can$ to $g_1=\bar g$. Next, apply Lemma~\ref{lem:gaussian-split-polynomial-cone}. We obtain a coefficient $\alpha>1$ for the quadratic potential, the cone $C=(0,\infty)\times\bbS^n$ equipped with the metric $g_C=\ud r^2+s^2r^2\bar g$ and the potential $V_C=\alpha r^2/2$, and a
vector space $E_C$ of polynomial test functions satisfying $\dim E_C=N_q+1$ and
$\sup_{0\neq p\in E_C}\cR_{g_C,V_C}(p)<q$. Here
$N_q=\binom{n+q}{n+1}$ is the number of eigenvalues of the standard Gaussian space, counted with multiplicity, which are strictly less than $q$.
Consequently,
\begin{align}\label{eq:gaussian-split-final-Gaussian-index}
        \lam_{N_q+1}
        (\bbR^{n+1},|\cdot|,\gamma^{n+1})
        =
        q.
\end{align}
Now apply Lemma~\ref{lem:gaussian-split-smoothing} and fix $\eta>0$ sufficiently
small. Set
\begin{align}
        g
        &:=
        g_\eta,
        \qquad
        V
        :=
        V_\eta,
        \qquad
        Z
        :=
        \int_{\bbR^d}e^{-V}\dvol_g,
        \notag\\
        \mu
        &:=
        Z^{-1}e^{-V}\dvol_g,
        \qquad
        K
        :=
        N_q+1.
\end{align}
Since $d=n+1$, Lemma~\ref{lem:gaussian-split-smoothing} shows that $g$ is a
smooth complete metric on $\bbR^d$, $V$ is smooth, $0<Z<\infty$, and
\begin{align}
        \Ric_g
        &\geq
        0,
        \qquad
        \nabla_g^2V
        \geq
        g,
        \notag\\
        \dim E_\eta
        &=
        K,
        \qquad
        \sup_{0\neq f\in E_\eta}
        \cR_{g,V}(f)
        <
        q.
\end{align}
Since $\mu=e^{-W}\dvol_g$, where $W:=V+\log Z$, we have
\begin{align}
        \Ric_g+\nabla_g^2W
        =
        \Ric_g+\nabla_g^2V
        \geq
        g.
\end{align}
Thus $\mu$ is a probability measure and $(\bbR^d,g,\mu)$ satisfies
$\CD(1,\infty)$. Finally,
\begin{align}
        \lam_K(\bbR^d,g,\mu)
        \leq
        \sup_{0\neq f\in E_\eta}
        \frac{
        \int_{\bbR^d}|\nabla f|_g^2\,\ud\mu
        }{
        \int_{\bbR^d}f^2\,\ud\mu
        }
        =
        \sup_{0\neq f\in E_\eta}
        \cR_{g,V}(f)
        <
        q
        =
        \lam_K(\bbR^d,|\cdot|,\gamma^d).
\end{align}
The first inequality follows from the min--max characterization and
$\dim E_\eta=K$. The second equality follows because the normalization factor
$Z^{-1}$ cancels in the Rayleigh quotient, while the last equality follows
from \eqref{eq:gaussian-split-final-Gaussian-index}, $d=n+1$, and
$K=N_q+1$. This proves \eqref{eqn:index-gaussian-split-K} and completes the
proof.
\end{proof}
\begin{cor}\label{cor:gaussian-split-main}
Let $(\R^d,g,\mu)$ be the weighted manifold constructed in Theorem
\ref{thm:gaussian-split-main}, where
$\mu=Z^{-1}e^{-V}\dvol_g$, $\Ric_g\geq0$, and
$\nabla_g^2V\geq g$. Then there does not exist a $1$-Lipschitz map
\[
        T:(\R^d,|\cdot|,\gamma^d)\to(\R^d,g,\mu)
\]
pushing forward $\gamma^d$ onto $\mu$ up to a finite constant.
\end{cor}

\begin{proof}
If such a map existed, then Theorem~\ref{thm:contraction} would imply that
for all $k\geq1$ we have
\begin{align}
        \lam_k(\R^d,g,\mu)
        \geq
        \lam_k(\R^d,|\cdot|,\gamma^d).
\end{align}
This contradicts \eqref{eqn:index-gaussian-split-K}.
\end{proof}

\bibliographystyle{alpha}
\bibliography{refs}
\bigskip
\centerline{\scshape Shrey Aryan}
\smallskip
{\footnotesize
 \centerline{Department of Mathematics, Massachusetts Institute of Technology}
\centerline{77 Massachusetts Avenue
Cambridge, MA 02139-4307, USA}
\centerline{\email{shrey183@mit.edu}}
}

\end{document}